\documentclass{aims_2}
\usepackage{amsmath}
  \usepackage{paralist}
  \usepackage{graphics} 
  \usepackage{epsfig} 
\usepackage{graphicx}  \usepackage{epstopdf}
 \usepackage[colorlinks=true]{hyperref}

\usepackage{amssymb,mathrsfs,subfigure}
\usepackage{amsfonts,amssymb,enumerate,bigstrut}
\usepackage {tikz}
\usepackage{colortbl}
\usepackage{float}

\hypersetup{urlcolor=blue, citecolor=red}

\newtheorem{theorem}{Theorem}[section]
\newtheorem{lemma}{Lemma}[section]

\newtheorem{proposition}{Proposition}[section]
\newtheorem{rmk}{Remark}[section]

\newtheorem{definition}{Definition}[section]

\renewcommand{\theequation}{\thesection.\arabic{equation}}

\newcommand{\R}{\mathbb{R}}
\newcommand{\N}{\mathbb{N}}
\newcommand{\Z}{\mathbb{Z}}

\newcommand{\sign}{\mathrm{sgn}}
\newcommand{\tv}{\mathrm{TV}\,}
\renewcommand{\AA}{\mathcal{A}}

\newcommand{\CC}{\mathcal{C}}
\renewcommand{\O}{\mathcal{O}}

\newcommand{\T}{\mathcal{T}}
\newcommand{\Id}{I_{2N}}

\newcommand{\DX}{{\Delta x}}
\newcommand{\DT}{{\Delta t}}
\def \ssigma {{\boldsymbol {\sigma} }}
\def \ww {{\boldsymbol {w} }}
\def \ee {{\rm e}}

\def \dd {{\boldsymbol {d} }}
\def \ddone {{\boldsymbol {d}_{\bf 1}}}
\def \gg {{\boldsymbol {\gamma} }}
\def \GG {{\boldsymbol {G} }}
\def \vv {{\boldsymbol {v} }}

\def\zero {{\boldsymbol{0}}}

\DeclareMathOperator*{\essinf}{ess\,inf}
\DeclareMathOperator*{\esssup}{ess\,sup}

\title[Decay for the damped wave equation] 
      {On the decay in $W^{1,\infty}$ for the 1D semilinear damped wave equation on a bounded domain}

\author[Amadori and Aqel]{}


 \email{debora.amadori@univaq.it}
 \email{fatimaaqel@najah.edu}

\thanks{Partially supported by Miur-PRIN 2015, "Hyperbolic Systems of Conservation Laws and Fluid Dynamics: Analysis and Applications", 
\# Grant No. 2015YCJY3A\_003, and by 2018 INdAM-GNAMPA Project "Equazioni iperboliche e applicazioni".}

\begin{document}


\maketitle

\centerline{\scshape Debora Amadori}
\medskip
{\footnotesize
 \centerline{Dipartimento di Ingegneria e Scienze dell'Informazione e Matematica (DISIM)}
   \centerline{University of L'Aquila, L'Aquila, Italy}
} 

\bigskip

\centerline{\scshape Fatima Al-Zahr\`a Aqel}
\medskip
{\footnotesize
 \centerline{Department of Mathematics, An-Najah National University}
   \centerline{Nablus, Palestine}
}

\bigskip\medskip


\begin{abstract} In this paper we study a $2\times2$ semilinear hyperbolic system of partial differential equations, which is related to a semilinear wave equation with nonlinear, time-dependent damping in one space dimension. 
For this problem, we prove a well-posedness result in $L^\infty$ in the space-time domain $(0,1)\times [0,+\infty)$. 
Then we address the problem of the time-asymptotic stability of the zero solution and show that, under appropriate conditions, 
the solution decays to zero at an exponential rate in the space $L^{\infty}$. The proofs are based on the analysis of the invariant domain
of the unknowns, for which we show a contractive property. These results can yield a decay property in $W^{1,\infty}$ for the corresponding solution to the semilinear wave equation.
\end{abstract}


\section{Introduction}
In this paper we study the initial--boundary value problem for the $2\times 2$ system in one space dimension
\begin{equation}
\begin{cases}
\partial_t\rho +  \partial_x J  = 0, &\\
\partial_t J  +  \partial_x \rho = - 2 k(x)\alpha(t) g(J), &
\end{cases} \label{DWE-rho-J-IBVP}
\end{equation}
where $x\in I\, =\, [0,1]$, $t\ge 0$ and 
\begin{equation}\label{init-boundary-data}
(\rho,J)(\cdot,0) =(\rho_0, J_0)(\cdot)\,,  \qquad \qquad J(0,t)= J(1,t)=0
\end{equation}
for $(\rho_0, J_0)\in L^\infty(I)$. About the terms $k$, $\alpha$ and $g$ in \eqref{DWE-rho-J-IBVP}, let 
\begin{equation*}
k\in L^1(I)\,,\quad k\ge 0\ a.e.\,, \qquad g\in C^1(\R)\,, \quad  g(0)=0\,, \quad g'(J)\ge 0
\end{equation*}
and
\begin{equation*}
\alpha\in BV_{loc}\cap L^\infty ([0,\infty); [0,1])\,,\qquad \alpha(t)\ge 0\,.
\end{equation*}
The problem \eqref{DWE-rho-J-IBVP}--\eqref{init-boundary-data} is related to the one-dimensional damped semilinear wave equation 
on a bounded interval: if $(\rho,J)(x,t)$ is a solution to~\eqref{DWE-rho-J-IBVP}, \eqref{init-boundary-data}, then
$$u(x,t) \dot = - \int_0^x \rho(y,t)\,dy$$
formally satisfies 
\begin{equation}\label{DWE}
u_x=-\rho\,,\qquad  u_t=J\,,\qquad \partial_{tt} u - \partial_{xx} u + 2 k(x) \alpha(t)g(\partial_t u)=0\,.
\end{equation}
In the time-independent case, $\alpha(t)=const.$, the large time behavior of solutions to \eqref{DWE-rho-J-IBVP}--\eqref{init-boundary-data} 
is governed by the stationary solution 
$$
J(x)=0, \qquad   \rho(x)= const. = \int_I \rho_0 \,.
$$ 
After possibly changing the variable $\rho$ with $\rho- \int_I\rho_0$, it is not restrictive to assume that $\int_I \rho_0(x)\,dx=0$\,.

\smallskip\par\noindent
The coefficient $\alpha(t)$ in \eqref{DWE-rho-J-IBVP}, with values in $[0,1]$, plays the role of a time localization of the damping term.
A specific time dependent case is the \textit{intermittent damping} \cite{MV2002,HMV2005}, in which for some $0<T_1<T_2$ one has 
\begin{equation}\label{hyp-on-alpha_ON-OFF}
\alpha(t) =\begin{cases}
1 & t\in[0,T_1), \\
0 \, & t\in [T_1,T_2) 
\end{cases}
\,,\qquad
\alpha(t+T_2)=\alpha(t)\quad \forall\,t>0\,. 
\end{equation}
The damped wave equation and its time-asymptotic stability properties have been studied in several papers, 
see for instance \cite{Z-Sirev-2005} and references therein, in terms of the decay of energy ($L^2$ norm of the derivatives of $u$). 
The $L^p$ framework, with $p\in [2,\infty]$ was considered in \cite{Haraux09,A-A-DS2018,CMP19}.

In this paper we continue the project, that was started in \cite{A-A-DS2018}, in two directions:

- first, we prove a well-posedness result, global in time, for the initial-boundary value problem \eqref{DWE-rho-J-IBVP}--\eqref{init-boundary-data} together with $L^\infty$ initial data; in turn, this result provides a well-posedness result in $W^{1,\infty}$ for the equation \eqref{DWE}. See Theorem~\ref{theorem:well-posedness};

- second, we address the time-asymptotic stability of the solution $\rho=0=J$; by following the approach introduced in \cite{A-A-DS2018}, we obtain a result on the exponential decay of the $L^\infty$--norm of the solution to \eqref{DWE-rho-J-IBVP}, under the assumption that the damping term is linear and time-independent; see Theorem~\ref{main-theorem-2}.  
In this specific context, this result extends the main result obtained in \cite{A-A-DS2018}, where $BV$ (Bounded Variation) initial data were assumed; since the constant values in the time-asymptotic estimate were depending on the total variation of the solution, a density argument was not sufficient to extend the result to the class of $L^\infty$ initial data.

\subsection{Main results}
We introduce the main results of this paper. The first one (Theorem~\ref{theorem:well-posedness}) concerns the existence and 
stability of weak solutions to \eqref{DWE-rho-J-IBVP} with time-dependent source, while the second one (Theorem~\ref{main-theorem-2}) 
concerns the asymptotic-time decay in $L^\infty$ of the solution under more specific assumptions.

\smallskip
We use the standard notation $\R_+ = [0,+\infty)$.

\begin{definition}\label{def:weak-sol}
Let $I\, =\, [0,1]$ and $(\rho_0, J_0)\in L^\infty(I)$. A weak solution of the problem \eqref{DWE-rho-J-IBVP}--\eqref{init-boundary-data} is a function
\begin{equation*}
(\rho,J):I\times \R_+ \to \R^2
\end{equation*}
that satisfies the following properties:
\begin{itemize}
\item[(a)] the map $\R_+ \ni t\mapsto (\rho,J)(\cdot,t)\in L^\infty(I)\subset L^1(I)$ is continuous with respect to the $L^1$--norm, 
and it satisfies 
$$(\rho,J)(\cdot,0) = (\rho_0, J_0);$$
\item[(b)] the equation \eqref{DWE-rho-J-IBVP}$_1$ is satisfied in the distributional sense in $[0,1]\times (0,\infty)$, while 
the equation \eqref{DWE-rho-J-IBVP}$_2$ in the distributional sense in $(0,1)\times (0,\infty)$\,.
\end{itemize}
\end{definition}
\par\noindent
The boundary condition in \eqref{init-boundary-data} is taken into account by means of the first part of (b), that is,
by requiring that for all functions $\phi\in C^1([0,1]\times (0,+\infty))$, with compact support in $[0,1]\times (0,+\infty)$, one has
\begin{align*}
\int_0^1\int_{0}^{\infty} \left\{ \rho \partial_t \phi + J \partial_x \phi\right\}\,dxdt=0\,.
\end{align*}
Now we state the following well-posedness result.
\begin{theorem}\label{theorem:well-posedness}
Assume that
\begin{equation}\label{hyp-weaker-on-k-and-g}
k\in L^1(I)\,,\quad k\ge 0\ a.e.\,, \qquad g\in C^1(\R)\,, \quad  g(0)=0\,, \quad g'(J)\ge 0
\end{equation}
and that
\begin{equation}\label{hyp-on-alpha}
\alpha\in BV_{loc}\cap L^\infty ([0,\infty); [0,1])\,,\qquad \alpha(t)\ge 0\,.
\end{equation}
Let $(\rho_0, J_0)\in L^\infty(I)$ with $\int_I \rho_0=0$. Then there exists a unique function
\begin{equation*}
(\rho,J):I\times \R_+ \to \R^2
\end{equation*}
which is a weak solution of \eqref{DWE-rho-J-IBVP}--\eqref{init-boundary-data} in the sense of Definition~\ref{def:weak-sol}. One has that
\begin{itemize}
    \item {\rm Conservation of mass}:
\begin{equation}\label{eq:int-rho-const}
\int_I \rho(x,t)\,dx=0\qquad \forall\, t>0\,.
\end{equation}
\item {\rm Invariant domain}: define the diagonal variables
\begin{equation}\label{diag-var-main-thm}
f^+ = \frac{\rho+J}2\,,\qquad f^- = \frac{\rho-J}2 
\end{equation}
and
\begin{equation}\label{def:inv-dom-main-thm}
M= \esssup_I{f_0^\pm}\,,\qquad m = \essinf_I{f_0^\pm}\,,
\end{equation}
\begin{equation}\label{D-D_J-def}
D=[m,M]\times[m,M]\,,  \qquad  
\qquad D_J= [- (M-m), M-m]\,.
\end{equation}
Then 
$D$, $D_J$ are invariant domains for $\left(\rho,J\right)$ and for $J$, respectively, in the sense that
\begin{equation*}
    m\le f^\pm (x,t)\le M\,,\qquad |J(x,t)|\le M-m\qquad a.e.\,.
\end{equation*}
\end{itemize}
\end{theorem}
\smallskip\par\noindent
Next, we consider the case of linear damping, 
that is for $k(x)$ and $\alpha(t)$ constant, $g(J)$ linear.
In the next theorem we establish a contractive property of the invariant domain when passing from $t=0$ to $t=1$.

\begin{theorem} \label{main-theorem-2} For $d>0$, consider the system 
\begin{equation}
\begin{cases}
\partial_t\rho +  \partial_x J  = 0, &\\
\partial_t J  +  \partial_x \rho = - 2 d J, &
\end{cases} \label{DWE-rho-J-IBVP-linear}
\end{equation}
where $x\in I$, $t\ge 0$, together with initial and boundary conditions \eqref{init-boundary-data},
$(\rho_0, J_0)\in L^\infty(I)$, and $\int_I \rho_0=0$.

Then there exists $d^*>0$ and a constant $\CC(d)$ depending on $d$ that satisfies
\begin{equation}
    0<\CC(d)<1\,,\qquad d\in(0,d^*)\,,
\end{equation}
such that the following holds:
\begin{equation}\label{eq:contraction-M1m1}
\esssup_I{f^\pm(x,t)} - \essinf_I{f^\pm(x,t)} \le  \CC(d)
\left(\esssup_I{f_0^\pm} - \essinf_I{f_0^\pm}\right)
\qquad \forall~ t\ge 1\,.
\end{equation}



 \end{theorem}
 
In other words, the estimate \eqref{eq:contraction-M1m1} indicates that the solution trajectory 
$t\mapsto f^\pm(\cdot,t)$, whose values belong to the invariant domain $D=[m,M]^2$ for $t\ge 0$ as in Theorem~\ref{theorem:well-posedness}, 
is contained in a smaller domain after time $T=1$. The new invariant domain is defined by $D_1=[m_1,M_1]^2$, where
\begin{equation}\label{def:M1-m1}
M_1= \esssup_I{f^\pm(x,1)}\,,\qquad m_1 = \essinf_I{f^\pm(x,1)},
\end{equation}
and the following properties hold:
\begin{equation*}
m\le m_1\le 0\le M_1\le M\,,\qquad M_1-m_1 \le \CC(d) (M-m) < M-m \qquad 0<d<d^*\,.
\end{equation*}
For the definition of $\CC(d)$ see \eqref{def:CC}.

\par\smallskip
As an application of Theorem~\ref{main-theorem-2}, we show two decay estimates for the linear system
\begin{equation}
\begin{cases}
\partial_t\rho +  \partial_x J  = 0, &\\
\partial_t J  +  \partial_x \rho = - 2 d \alpha(t) J\,. &
\end{cases} \label{DWE-linear-alpha}
\end{equation}


\begin{theorem}\label{main-theorem-3-applications}
For $d\in(0,d^*)$, consider the system \eqref{DWE-linear-alpha} where $x\in I$, $t\ge 0$, 
together with initial and boundary conditions \eqref{init-boundary-data},
$(\rho_0, J_0)\in L^\infty(I)$, and $\int_I \rho_0=0$.
\begin{itemize}
    \item[\textbf{(a)}]  If $\alpha(t)\equiv 1$, there exist constant values $C_j>0$,  $j=1,2,3$, that depend only on the equation and on
the initial data, such that 
    \begin{align}\label{decay-J-rho}
\|J(\cdot,t)\|_{L^\infty}&\leq C_1 \ee^{-C_3t}\nonumber  \\
\|\rho(\cdot,t)\|_{L^\infty}&\leq C_2 \ee^{-C_3t}
\end{align}
with 
    \begin{equation*}
        C_3 = |\ln\left(\CC(d)\right)|\,.
    \end{equation*}

    \item[\textbf{(b)}] For $\alpha(t)$ of type "on-off" as in \eqref{hyp-on-alpha_ON-OFF}, with $T_1\ge 1$, one has \eqref{decay-J-rho} with 
    \begin{equation*}
        C_3 = \frac{[T_1]}{T_2} |\ln\left(\CC(d)\right)|
    \end{equation*}
    where $[T_1]\ge 1$ denotes the integer part of $T_1$.
\end{itemize}
\end{theorem}

In addition to the previous statement, if $(\rho_0, J_0)\in BV(I)$, then the approximate solutions $(\rho^\DX,J^\DX)(x,t)$ of \eqref{DWE-rho-J-IBVP-linear}, 
\eqref{init-boundary-data} as defined in Section~\ref{subsec:approximate} satisfies the $L^\infty$ error estimate \eqref{estim:fpm-time-h} 
established in Theorem~\ref{th:Linfty-d}\,.

\begin{rmk} Some final remarks are in order.

(a) In terms of the damped wave equation \eqref{DWE}, Theorem~\ref{main-theorem-3-applications} can yield a result on the decay in $W^{1,\infty}$ of the solution $u$ towards zero. Indeed the function
$$
u(x,t) ~\dot = \int_0^x \rho(x',t)\, dx'\,, \qquad x\in (0,1)
$$
is Lipschitz continuous in $x$, satisfies $u(0,t)=u(1,t)=0$ because of \eqref{eq:int-rho-const} and
$$
\|u(\cdot,t)\|_\infty \le  \|\rho(\cdot,t)\|_\infty\,. 
$$
Hence if $\rho(\cdot,t)$ converges to 0 in $L^{\infty}$, then $u(\cdot,t)$ converges to 0 in $W^{1,\infty}$. 

For a rigourous proof of a decay estimate for the semilinear wave equation, one should prove that such $u \in C^0\left(\R_+; H_0^1(I)\right)\times C^1\left(\R_+; L^2(I)\right)$ and that it is a solution of \eqref{DWE} together with boundary conditions
$u(0,t)=u(1,t)=0$ and initial conditions
$$
u(x,0) = u_0(x) = \int_0^x \rho_0(x')\, dx'\,,\qquad \partial_t u(x,0) = J_0(x)\,.
$$

\smallskip
(b) The result in Theorem~\ref{main-theorem-3-applications}, case \textbf{(a)} extends readily to the case of non-zero, constant boundary conditions for $J$. Indeed consider the system \eqref{DWE-rho-J-IBVP} together with initial data $(\rho_0, J_0)\in L^\infty(I)$ and boundary conditions
\begin{equation}\label{eq:non-zero-const-bc-J}
J(0,t)= J(1,t)=\beta \in \R\,.    
\end{equation}
Let's define
\begin{equation*}
\rho_\beta(x)  = - 2 g(\beta) \int_0^x k(y)\,dy + C 
\end{equation*}
where the constant $C$ is identified uniquely by the property of conservation of mass:
\begin{equation*}
    \int_0^1  \rho_\beta(x)\,dx = \int_0^1 \rho_0(x)\,dx\,.
\end{equation*}
If $\alpha(t)\equiv 1$, then the change of variables
\begin{equation*}
v=\rho- \rho_\beta,\qquad w=J -  \beta\,,
\end{equation*}
on the system \eqref{DWE-rho-J-IBVP} yields 
\begin{equation}\label{eq:homog}
\begin{cases}
\partial_t v +  \partial_x w  = 0 &\\
\partial_t  w+  \partial_x v = - 2  k(x) \widetilde g (w; \beta) &\qquad \widetilde  g  (w; \beta)=  g( \beta + w) - g(\beta)
\end{cases}
\end{equation}
together with initial-boundary conditions
\begin{equation*}
(v,w)(\cdot,0) =(\rho_0 - \rho_\beta, J_0 - \beta)(\cdot)\,,  \qquad \qquad w(0,t)= w(1,t)=0
\end{equation*}
where $w\mapsto  \widetilde g  (w; \beta)$ has the same properties of $g$ in \eqref{hyp-weaker-on-k-and-g} with $\sup g' = \sup \widetilde g'$ 
on corresponding bounded domains, and $\int_I v_0\,dx =0\,.$ Therefore a decay estimate for $J(\cdot, t) - \beta$, 
$\rho(\cdot, t) - \rho_\beta(\cdot)$ holds as in \eqref{decay-J-rho}. On the other hand,  in the on-off case \textbf{(b)} with boundary conditions \eqref{eq:non-zero-const-bc-J} and $\beta\not =0$, 
the non--constant function $\rho_\beta(x)$ is no longer stationary and the long time behavior of $(\rho,J)(\cdot,t)$ requires further investigation.
\end{rmk}

\smallskip\par\noindent
The paper is organized as follows. In Section~\ref{Sec:2} we recall some preliminaries on Riemann problems for a hyperbolic system
which is a $3\times3$ extended version of \eqref{DWE-rho-J-IBVP}, and prove interaction estimates that take into account of the time change 
of the damping term.
In Section~\ref{sec:approximate} we provide the proof of Theorem~\ref{theorem:well-posedness} by following the approach considered in 
\cite{A-A-DS2018}, which is readily adapted to the time-varying source term of the system~\eqref{DWE-rho-J-IBVP}. In section~\ref{sec:iter-matrix-discrete}, we study the representation of the approximate solution which turns out to be a vector representation, see Lemma~\ref{prop:representation-J-rho}. In Section~\ref{sec:linear-case}, we prove Theorem~\ref{main-theorem-2} and, finally,
in Section~\ref{sec:6} we prove Theorem~\ref{main-theorem-3-applications}.

%
%

\section{Preliminaries}\label{Sec:2}
\setcounter{equation}{0}

In terms of the diagonal variables $f^\pm$, defined by
\begin{equation}\label{diag-var}
\rho=f^+ + f^-\,,\qquad  J=f^+ - f^-
\end{equation}
the system~(\ref{DWE-rho-J-IBVP}) rewrites as a discrete-velocity kinetic model
\begin{equation}
\begin{cases}
\partial_t f^- -  \partial_x f^- = {k(x) \alpha(t)} \,g(f^+ - f^-),  &
\\
\partial_t f^+ +  \partial_x f^+  = -  {k(x)\alpha(t)}\, g(f^+ - f^-) \,. &
\end{cases} \label{GT}
\end{equation}
\subsection{The time-independent case: the Riemann problem}
In the following we assume that $\alpha(t)\equiv 1$. Then \eqref{DWE-rho-J-IBVP} and \eqref{GT} can be rewritten, respectively, as
\begin{equation}
\begin{cases}
\partial_t\rho +  \partial_x J  & = 0\,, \\
\partial_t J  +  \partial_x \rho + 2 g(J)  \partial_x a & =0\,,  \\
\partial_t a &=0\,,
\end{cases}\qquad a(x)=\int_0^x k(y)\,dy  \label{DWE-rho-J-a}
\end{equation}
and
\begin{equation}
\begin{cases}
\partial_t f^- - \partial_x f^- - g(f^+ - f^-)\partial_x a &=0\,, \\
\partial_t f^+ + \partial_x f^+ + g(f^+ - f^-)\partial_x a &=0 \,, \\
\partial_t a &=0\,. 
\end{cases} \label{NC-system}
\end{equation}
The characteristic speed of system \eqref{NC-system} are $\mp 1, 0$. 
We call \textit{$0$-wave curves} those characteristic curves corresponding to the speed $0$; they are related to the stationary equations for $f^\pm$,
that is 
\begin{equation}\label{eq:stationary-fpm}
\partial_x f^\pm = - g(f^+ - f^-)\partial_x a\,.
\end{equation}
We denote either by $(\rho_\ell,J_\ell,a_\ell)$,  $(\rho_r,J_r,a_r)$  or by $(f^-_\ell,f^+_\ell,a_\ell)$, $(f^-_r,f^+_r,a_r)$ 
the left and right states corresponding to Riemann data for \eqref{DWE-rho-J-a},  \eqref{NC-system} respectively.
\begin{proposition}\cite{AG-MCOM16}\label{prop:1}
Assume that $k(x)\ge 0$, $g(J)J\ge 0$ and consider the initial states 
$$
U_\ell=(\rho_\ell,J_\ell,a_\ell)\,,\qquad U_r=(\rho_r,J_r,a_r)
$$  
with corresponding states $(f^-_\ell,f^+_\ell,a_\ell)$\,, $(f^-_r,f^+_r,a_r)$ in the $(f^\pm,a)$ variables.
Assume $a_\ell\le a_r$ and set 
\begin{equation}\label{def:delta}
\delta ~\dot = ~a_r - a_\ell\ge 0\,. 
\end{equation}
Then the following holds.

\begin{itemize}
\item[(i)]  The solution to the Riemann problem for system \eqref{DWE-rho-J-a} and initial data $U_\ell,U_r$
is uniquely determined by
\begin{equation}\label{sol-RP}
U(x,t) = 
\begin{cases}
U_\ell& x/t<-1\\
U_*=(\rho_{*,\ell}, J_*,a_\ell) & -1<x/t<0\\
U_{**}=(\rho_{*,r}, J_*,a_r) & 0<x/t<1\\
U_r & x/t> 1
\end{cases}
\end{equation}
with 
\begin{equation}\label{J*_rho*}
J_*+ g(J_*)\delta = f^+_\ell - f^-_r\,,\qquad \rho_{*,r}-\rho_{*,\ell}= -2g(J_*)\delta\,,
\end{equation}
see Figure~\ref{fig:RP}. 

\vspace{3pt}

\item[(ii)] If $m<M$ are given real numbers, the square $[m,M]^2$ is invariant for the solution to the Riemann problem in the $(f^-,f^+)$-plane. 
That is, the solution $U(x,t)$ given in \eqref{sol-RP} satisfies 
\begin{equation}\label{eq:inv-domain}
f^\pm (x,t) \in  [m,M]  
\end{equation}
for any $(f^-_\ell,f^+_\ell)$,  $(f^-_r,f^+_r) \in [m,M]^2$ and for any $\delta\ge 0$.

\vspace{3pt}

\item[(iii)] For every pair $U_\ell$, $U_r$ 
with $(f^-_\ell,f^+_\ell)$,\  $(f^-_r,f^+_r) \in [m,M]^2$, let $\sigma_{-1} = (J_* - J_\ell)$ and  $\sigma_{1} = (J_r-J_*)$.
Hence, 
\begin{equation}\label{ineq:sizes}
\left| |\sigma_1| - |f^+_r - f^+_\ell|\right| \le C_0 \delta\,,\qquad \left| |\sigma_{-1}| - |f^-_r - f^-_\ell|\right| \le C_0 \delta\,,
\end{equation}
where 
\begin{equation}\label{def:C0}
C_0 = \max\{g(M-m), - g(m-M)\}\,.
\end{equation}
\end{itemize}
\end{proposition}

\begin{figure} 
\centering
\begin{tikzpicture} 

\draw[thick] (2.5,0.5) -- (2.5,3);
\draw[thick] (0,0.5) -- (5,0.5);
\draw[thick] (2.5,0.5) -- (0,3);
\draw[thick] (2.5,0.5) -- (5,3);

\node () at (1,1) {$U_\ell$};
\node () at (4,1) {$U_r$};
\node () at (1.8,2) {$U_*$};
\node () at (3.2,2) {$U_{**}$};

\node () at (4.3,2.7) {$\sigma_1$};
\node () at (0.7,2.7) {$\sigma_{-1}$};
\node () at (2.2,2.7) {$\delta$};

\node () at (2.5,0.2) {$0$};
\end{tikzpicture}
\caption{Structure of the solution to the Riemann problem.}\label{fig:RP}
\end{figure}
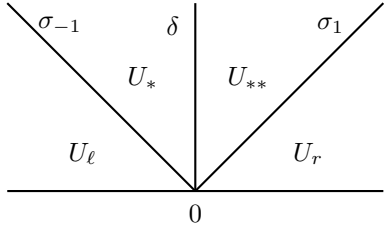

We stress that, in \eqref{ineq:sizes}--\eqref{def:C0}, the quantity $C_0$ is independent of $\delta\ge 0$. 

Here and in the following, we denote by $\Delta \phi(x)$ the difference $\phi(x+) - \phi(x-)$, where $\phi$ is a real-valued function 
defined on a subset of $\R$, and the limits $\phi(x\pm)= \lim_{y\to x\pm}\phi(y)$ exist. 

\smallskip\par
We define the amplitude of $\pm1$--waves as follows:
\begin{equation}\label{def-sizes}
\sigma_{\pm1} =  \Delta J = \pm \Delta f^\pm = \pm \Delta \rho \,. 
\end{equation}
In particular, with the notation of Figure~\ref{fig:RP}, we have
\begin{equation*}
\begin{split}
J_r-J_\ell
 &=  \sigma_{1}+ \sigma_{-1} \\
\rho_r-\rho_\ell &= \sigma_{1} - \sigma_{-1} -2g(J_*)\delta\,.
\end{split}
\end{equation*}

\subsection{The time-dependent case: interaction estimates}
As time evolves, the wave-fronts that stem from $t=0$ propagate and interact between each other; also the coefficient $\alpha(t)$ changes in time.
In order to get a-priori estimates on their total variation and $L^\infty$--norm, we study the interactions of waves in the solutions 
to \eqref{NC-system}. 

In \cite[Proposition~3]{A-A-DS2018}, the multiple interaction of two $\pm1$ waves with a single $0$--wave of size $\delta=a_r - a_\ell>0$ is studied.
The following proposition extends such a statement to the case in which the time dependent coefficient $\alpha(t)$ has a jump at the time of the interaction. We clarify that the values of $a_\ell$ and $a_r$, respectively on the left and on the right of the $0$--wave, do not change across the interaction; this is related to the third equation in \eqref{NC-system}.

\begin{proposition} (Multiple interactions, time-dependent case)\label{prop:multiple} 
Assume that at a time $\bar t>0$ an interaction involving a $(+1)$--wave, a $0$--wave and a $(-1)$--wave  occurs, see Figure~\ref{fig:multiple}.
Let  $\delta$ be as in \eqref{def:delta} and $\alpha^\pm \ge 0$ be given, so that $\alpha(t) = \alpha^+$ for $t>\bar t$ and  
$\alpha(t) = \alpha^-$ for $t<\bar t$. Assume that 
\begin{equation}\label{A-less-than-1}
(\sup g') \delta \alpha^\pm <1\,.
\end{equation}

Let  $\sigma^-_{\pm1}$ be the sizes (see \eqref{def-sizes}) of the incoming waves and $\sigma^+_{\pm1}$ be the sizes of the outgoing ones.
Let $J^\pm_*$ be the intermediate values of $J$ (which are constant across the $0$--wave), before and after the interaction 
as in Figure~\ref{fig:multiple}, and choose a value $s\in (\min{J^\pm_*}, \max{J^\pm_*})$ such that
\begin{equation}\label{property-of-s}
g'(s) = \frac{g(J^+_*)-g(J^-_*)}{J^+_* - J^-_*}\,.
\end{equation}
Then, for $\gamma^\pm ~\dot =~ g'(s)\delta \alpha^\pm$\,,  
it holds
\begin{equation}\label{mult-inter-matrix-form}
\begin{pmatrix}
\sigma^+_{-1}\\    \sigma^+_1
\end{pmatrix} = \frac{1}{1+\gamma^-}\begin{pmatrix}
1&\gamma^-\\
\gamma^-&1
\end{pmatrix} 
\begin{pmatrix}
\sigma^-_{-1}\\
\sigma^-_1
\end{pmatrix}
+(\alpha^+ - \alpha^-) \delta\,  \frac{g(J_*^+)}{1+\gamma^-}\begin{pmatrix}
-1\\
+1
\end{pmatrix}\,, 
\end{equation}
and similarly
\begin{equation}\label{mult-inter-matrix-form-gammaPLUS}
\begin{pmatrix}
\sigma^+_{-1}\\    \sigma^+_1
\end{pmatrix} = \frac{1}{1+\gamma^+}\begin{pmatrix}
1&\gamma^+\\
\gamma^+&1
\end{pmatrix} 
\begin{pmatrix}
\sigma^-_{-1}\\
\sigma^-_1
\end{pmatrix}
+  (\alpha^+ - \alpha^-)\delta \, \frac{ g(J_*^-) }{1+\gamma^+}\begin{pmatrix}
-1\\
+1
\end{pmatrix}\,.
\end{equation}
Moreover,
\begin{align}\label{prop:id1}
\sigma^+_1 +  \sigma^+_{-1} &= \sigma_1^- +  \sigma^-_{-1} \\
\label{eq:no-decay}
|\sigma^+_{-1}| + |\sigma^+_{1}|  &\le~  |\sigma^-_{-1}| + |\sigma^-_{1}| + 2 C_0 \delta |\alpha^+ - \alpha^-| 
\end{align}
with $C_0= \max\{g(M-m), - g(m-M)\}$ as in \eqref{def:C0}, together with
$$
m=\min \left\{f^{\pm}_{\ell}, f^{\pm}_{r}\right\}, \qquad M=\max  \left\{f^{\pm}_{\ell}\,, f^{\pm}_{r}\right\}\,.
$$
\end{proposition}
\begin{figure}
\centering
\begin{tikzpicture} 

\draw[thick] (2.5,0.5) -- (2.5,3.5);
\draw[black,fill=gray] (1,0.5) -- (2.5,2.0) -- (2.5,0.5);
\draw[black,fill=gray] (2.5,0.5) -- (2.5,2.0) -- (4,0.5);
\draw[black,fill=gray] (1,3.5) -- (2.5,2.0) -- (2.5,3.5);
\draw[black,fill=gray] (2.5,3.5) -- (2.5,2.0) -- (4,3.5);
\draw[black,dashed] (0,2) -- (5,2);

\draw (5,2.5) node[draw,fill=white,rounded corners] (a) {$J_*^+$};
\draw[line width=1pt,color=red,-stealth] (a.west) to[bend left] (2.7,3.2);

\draw (5,1.3) node[draw,fill=white,rounded corners] (a) {$J_*^-$};
\draw[line width=1pt,color=red,-stealth] (a.west) to[bend right] (2.7,1.0);

\node () at (-1,2) {$t=\bar t$};

\node () at (4.1,3.1) {$\sigma_1^+$};
\node () at (0.9,3.1) {$\sigma_{-1}^+$};

\node () at (0.9,1) {$\sigma_1^-$};
\node () at (4.1,1) {$\sigma_{-1}^-$};

\node () at (2.5,0.3) {$\delta \alpha^-$};
\node () at (2.5,3.7) {$\delta \alpha^+$};
\end{tikzpicture}
\caption{Multiple interaction, time-dependent case. 
}\label{fig:multiple}
\end{figure}
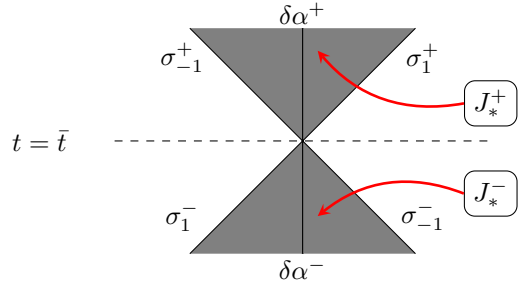
\begin{rmk} 
\begin{enumerate}[(a)]
\item If $\alpha(t)$ is as in \eqref{hyp-on-alpha_ON-OFF}, the \fbox{\sc{on--off}} time corresponds to $\alpha^- =1$,  $\alpha^+=0$ 
while the \fbox{\sc{off--on}} time corresponds to $\alpha^- =0$,  $\alpha^+=1$\,.
\item With the notation of Proposition~\ref{prop:multiple}, one has
\begin{equation}\label{inv-domain}
f^{\pm}_{*, \ell},~f^{\pm}_{*, r}   \in [m, M]\,,\quad |s|\le M-m
\end{equation} 
where $f^{\pm}_{*, \ell}, f^{\pm}_{*, r}$ are the intermediate states after the interaction time.

Indeed, as a consequence of Proposition~\ref{prop:1}--{\it (ii)}, the values $f^{+}_{*, \ell},~f^{+}_{*, r}$ belong to $[m, M]$. Using the same argument of the
proof of Proposition~\ref{prop:1} in \cite{AG-MCOM16}, one can conclude that the same property holds also for the intermediate state {\bf before} 
the interaction, that is,  $f^{-}_{*, \ell},~f^{-}_{*, r} \in [m, M]$. As a consequence, both the intermediate values $J_*^\pm$ satisfy
$$
|J_*^\pm|\le M-m
$$
and hence, by the intermediate value theorem used in \eqref{property-of-s}, we obtain that $|s|\le M-m$.
\end{enumerate}
\end{rmk}

\begin{proof}[Proof of Proposition~\ref{prop:multiple}]
Let $J_*^-$, $J_*^+$ be the intermediate values of $J$ before and after the interaction, respectively. 
By \eqref{J*_rho*} these values satisfy 
\begin{equation*}
J^+_* + g(J^+_*) \delta \alpha^+ =  f_\ell^+ - f_r^-\,, \qquad J^-_* - g(J^-_*) \delta  \alpha^- =  f_r^+ - f_\ell^-\,.
\end{equation*}
Since the quantity $J_r-J_\ell$ remains constant across the interaction, we get 
$$
J_r-J_\ell = (J_r- J^+_*) + (J^+_*-J_\ell) =  (J_r- J^-_*) + (J^-_*-J_\ell)\,.
$$
Then, by the definition \eqref{def-sizes} of the sizes ($\sigma_{\pm 1} = \Delta J$) we deduce the identity \eqref{prop:id1}.
Using again \eqref{J*_rho*} and \eqref{def-sizes}, the same procedure applied to $\rho_r-\rho_\ell$ and the fact that 
$\sigma_{\pm 1} = \pm \Delta \rho$ lead to the following identity:
\begin{equation*}
\sigma^+_1 - \sigma^+_{-1} - 2 g(J^+_*) \delta  \alpha^+ = \sigma_{1}^- -  \sigma^-_{-1} - 2 g(J^-_*) \delta  \alpha^-\,,
\end{equation*}
that can be rewritten as
\begin{align}\nonumber
\sigma^+_1 - \sigma^+_{-1} &= \sigma_{1}^- -  \sigma^-_{-1} + 2\left[ g(J^+_*) - g(J^-_*) \right]\delta  \alpha^- + 2 g(J^+_*) \delta (\alpha^+ - \alpha^-)\\
&= \sigma_{1}^- - \sigma^-_{-1} + 2 g'(s) \left[ J^+_* - J^-_* \right]\delta  \alpha^- + 2 g(J^+_*) \delta ( \alpha^+ -  \alpha^-) \label{ident-sigma-multi-line-3}
\end{align}
for $s$ as in \eqref{property-of-s}. Notice that 
\begin{equation*}
J^+_* - J^-_* = (J^+_* - J_r) + (J_r - J^-_*) = - \sigma_1^+ + \sigma_{-1}^-
\end{equation*}
and, replacing $J_r$ with $J_\ell$, one has 
\begin{equation*}
J^+_* - J^-_* =  \sigma_{-1}^+ - \sigma_{1}^-\,.
\end{equation*}
Since both equations are true, then one can combine them and write 
\begin{equation*}
J^+_* - J^-_* =  \frac12 \left(\sigma_{-1}^+ - \sigma_{1}^+ + \sigma_{-1}^-  - \sigma_{1}^-\right) \,.
\end{equation*}
By substitution into \eqref{ident-sigma-multi-line-3}, we get
\begin{align*}
\sigma^+_{1} - \sigma^+_{-1} &= \sigma_1^- -  \sigma_{-1}^-  + g'(s)\left(\sigma_{-1}^+ - \sigma_{1}^+ + \sigma_{-1}^- - \sigma_{1}^-\right)
\delta  \alpha^- + 2 g(J^+_*)\delta ( \alpha^+ -  \alpha^-)\,,
\end{align*}
which, for $\gamma^- ~\dot =~ g'(s)\delta  \alpha^-$ leads to 
\begin{equation*}
\left( 1+\gamma^- \right) \left(\sigma^+_{1} - \sigma^+_{-1}\right) = \left( 1- \gamma^- \right) 
\left( \sigma_1^- - \sigma_{-1}^- \right) + 2 g(J^+_*) \delta ( \alpha^+ -  \alpha^-)\,.
\end{equation*}
In conclusion, recalling \eqref{prop:id1}, we have the following $2\times2$ linear system
\begin{align*}
\sigma^+_1 +  \sigma^+_{-1} & = \sigma_1^- +  \sigma^-_{-1}  \nonumber
\\
\sigma^+_{1} - \sigma^+_{-1} &= \frac{1-\gamma^-}{1+\gamma^-}  \left( \sigma_1^- -  \sigma_{-1}^- \right) + \frac{2 g(J^+_*) 
\delta (\alpha^+ - \ \alpha^-)}{1+ \gamma^-} 
\end{align*}
whose solution is given by \eqref{mult-inter-matrix-form}. The proof of \eqref{mult-inter-matrix-form-gammaPLUS} is completely similar. 
Finally, by taking the absolute values in \eqref{mult-inter-matrix-form}, we get
\eqref{eq:no-decay}.
This concludes the proof of Proposition~\ref{prop:multiple}.
\end{proof}

\section{Approximate solutions and well-posedness}\label{sec:approximate}
\setcounter{equation}{0}

This section is devoted to the construction of a family of approximate solutions to the problem \eqref{DWE-rho-J-IBVP}, 
\eqref{init-boundary-data}. In Subsection~\ref{subsec:approximate} we will describe the algorithm, that follows the approach in \cite{A-A-DS2018},
while in Subsections~\ref{subsec:3.2}--\ref{subsec:3.3} we provide a-priori estimates on such approximations.

More generally, the approximation scheme follows the \emph{well-balanced} approach introduced in \cite{goto,laurent_Book} and employed in 
\cite{AG-MCOM16,AG-Briefs15,AG-AnIHP16} for the Cauchy problem. Also, the approximate solutions that are constructed here, 
are \emph{wave-front tracking} solutions (see \cite{Bressan_Book}) of the system~\eqref{DWE-rho-J-a} or, equivalently, \eqref{NC-system}.
  
Finally, in Subsection~\ref{subsec:convergence}, we prove the convergence of the approximate solutions in the $BV$ setting and use the 
stability in $L^1$, together with a density argument, to show the existence and stability for $L^\infty$ initial data $(\rho_0,J_0)$, thus completing 
the proof of Theorem~\ref{theorem:well-posedness}\,.  

\subsection{Approximate solutions}\label{subsec:approximate}
In this subsection, following \cite{A-A-DS2018}, we construct a family of approximate solutions for the initial--boundary value problem 
associated to system~\eqref{DWE-rho-J-a} and initial, boundary conditions \eqref{init-boundary-data} with $(\rho_0, J_0)\in BV(I)$ and 
\begin{equation}\label{eq:zero-mean-rho}
\int_I \rho_0(x)\,dx=0\,.
\end{equation}
Let $N\in 2\N$ and set
\begin{equation*}
\DX=\DT=\frac 1N\,,\qquad x_j = j\DX \ (j=0,\ldots,N) \,,\qquad t^n=n\DT\ (n\ge 0)\,.
\end{equation*}
The size of the $0$-wave at a point $0<x_j<1$ is given by
\begin{align}\label{delta-j}
\delta_j = \int_{x_{j-1}}^{x_{j}} k(x) dx\,, \qquad j=1,\ldots, N-1\,. 
\end{align}
Assume $\DX=1/N$ small enough so that  
\begin{equation}\label{delta-j-small}
\sup g'(J) \|\alpha\|_\infty  \cdot \delta_j <1\,.
\end{equation}
The functions 
$$
f_0^- = \frac12 \left(\rho_0 - J_0\right)\,,\qquad f_0^+ = \frac12 \left(\rho_0 + J_0\right)
$$
clearly belong to $BV(I)$. In terms of the system \eqref{NC-system}, we 
approximate the initial data $f_0^\pm$ and $a(x)$ as follows:
\begin{equation}\label{init-data-approx}
(f_0^\pm)^\DX(x) = f_0^\pm(x_j+)\,,\qquad     a^\DX(x) = a(x_j)=\int_0^{x_j}k \,, \qquad x\in(x_j,x_{j+1})\,. 
\end{equation}
Recalling that $\int \rho_0 \,dx =0$ and that $\rho=f^++f^-$, we easily deduce the following inequality:
\begin{equation}\label{int-rho-DX}
\left| \int_I \left[ (f_0^+)^\DX + (f_0^-)^\DX \right]\,dx \right| \le \DX \tv \rho_0\,. 
\end{equation}
Finally we approximate $\alpha(t)$ in a natural way as follows:
\begin{equation}\label{def:alpha_n}
\alpha_n(t) = \bar \alpha_n := \alpha(t^n+) \qquad \mbox{for  }t\in [t_n,t_{n+1})\,,\quad n\ge 0.
\end{equation}
Beyond the adaptation to the time-dependence of the source term in \eqref{DWE-rho-J-IBVP},
the construction is completely similar to the one in \cite[Section 3]{A-A-DS2018}, leading to the definition of an approximate solution
$(f^\pm)^\DX(x,t)$ and hence of $\rho^\DX$, $J^\DX$. In the rest of this section, as far as there is no ambiguity in the notation, 
we will drop the $\DX$ and will refer to $(f^\pm)(x,t)$ as an approximate solution with fixed parameter $\DX>0$.

\subsection{Invariant domains}\label{subsec:3.2}
Recalling Proposition~\ref{prop:1}-(ii), the set
\begin{equation}\label{def:inv-dom}
D=[m,M]\times[m,M]\,,  \qquad M= \esssup_I{f_0^\pm}\,,\quad m = \essinf_I{f_0^\pm}
\end{equation}
is an invariant domain for the solution to the Riemann problem in the $(f^-,f^+)$-variables. Let
\begin{equation}\label{D_J-def}
J_{\max}= M-m 
\,,\qquad D_J= [- J_{\max}, J_{\max}]\,.
\end{equation}
Here $D_J$ denotes the closed interval which is the projection of $D$ on the $J$-axis.

\smallskip
It is easy to verify that $D$ is invariant also under the solution to the Riemann problem at the boundary. 
Indeed, assume that there is a $(-1)$-wave impinging on the boundary $x=0$ at a certain time $\bar t$ with a $+1$ reflected wave.
Let $(\bar f^-, \bar f^+)\in D$ be the state on the right of the impinging/reflected wave.
Hence 

$\bullet$\quad the state between $x=0$ and the impinging wave, for $t<\bar t$, is $(\bar f^+, \bar f^+)$,

$\bullet$\quad the state between $x=0$ and the reflected wave, for $t>\bar t$, is $(\bar f^-, \bar f^-)$,

\smallskip\par\noindent
and both these states belong to $D$. Finally we claim that $m\le 0\le M$. Indeed, since $\int_I \rho_0 =0$, then 
$$\essinf \rho_0 \le 0  \le \esssup \rho_0\,.$$
Using the elementary inequalities 
$\max\{ x+y,x-y\} \ge x \ge \min \{ x+y,x-y\}$, and recalling 
that $f^\pm = (\rho\pm J)/2$, we deduce that
$$
2 \essinf f^\pm_0\le \essinf \rho_0 \le 0  \le  \esssup \rho_0
\le 2 \esssup f^\pm_0
$$
and hence the claim.

All these properties are summarized in the following proposition.

\smallskip\par\noindent
\begin{proposition} \label{prop:inv-domains}
Under the assumptions of Theorem~\ref{theorem:well-posedness}, one has that
\begin{equation}\label{bound-on-M-m}
m\le 0 \le M\,.
\end{equation}
Moreover for every $t\ge 0$ the following holds:
\begin{equation}\label{bound-on-fpm}
m\le f^\pm(x,t) \le M
\end{equation}
and hence, by means of \eqref{diag-var}, 
\begin{equation}\label{bound-on-rho-J}
2m \le \rho(x,t)\le 2M\,, \qquad |J(x,t)| \le M - m
\end{equation}
with $m$, $M$ given in \eqref{def:inv-dom}.
\end{proposition}

As a consequence of the properties above, the solution satisfies $J(x,t)\in D_J$ outside discontinuities.
\begin{rmk}\label{rem:inv-domain}
We remark that, given $m<M$, the bounds \eqref{bound-on-fpm}, \eqref{bound-on-rho-J} hold 

$\bullet$\quad  for every choice of source term coefficients $k(x)$, $g(J)$, $\alpha(t)$ as in \eqref{hyp-weaker-on-k-and-g}, \eqref{hyp-on-alpha};

$\bullet$\quad  for every (approximate) solution such that the initial data satisfies \eqref{init-data-approx} and the bounds 
$$
m \le  \essinf_I{f_0^\pm} \le  \esssup_I{f_0^\pm} \le M \,.
$$

\end{rmk}

\smallskip\noindent
We also remark that, in case of no source term (for instance if $k(x)\equiv0$), by the analysis of the Riemann problems one finds that 
the invariant domain is smaller than the square $D$, being the rectangle $[m^-,M^-]\times[m^+,M^+]$:
\begin{align*}
m^\pm \le f^\pm(x,t) \le  M^\pm\,,
\end{align*}
where 
$$
m^\pm\, \dot =\, \inf_I{f_0^\pm}\,,\qquad  M^\pm \, \dot = \,\sup_I{f_0^\pm}\,.
$$

\subsection{
Conservation of mass}
In this subsection we prove that the total mass of $\rho^\DX$ is conserved in time.

\begin{proposition} \label{prop:cons-of-mass} In the previous assumptions, one has
\begin{equation}\label{eq:deriv-int-rho-0}
\frac d{dt}  \int_I \rho^\DX(x,t)\,dx =  0 \,,
\end{equation}
and 
\begin{equation}\label{approximate-cons-mass}
\left| \int_I  \rho^\DX(x,t) \,dx \right| \le \DX \cdot \tv \rho_0\,. 
\end{equation}
\end{proposition}

\begin{proof}
Let  
\begin{equation}\label{def:y-j}
y_1(t) < y_2(t) < \ldots < y_{2N}(t)\qquad \forall \, t >0\,,\ t\not = t^n,\ t\not =t^{n+1/2}
\end{equation}
be the location of the $\pm1$ waves at time $t$, that is, the location of all the possible discontinuities
(see Figure~\ref{fig:illustration-sigmaj}).  

By the Rankine-Hugoniot condition of the first equation 
in \eqref{DWE-rho-J-IBVP}, which is satisfied \emph{exactly} in the approximate solution, we have
\begin{equation}\label{RH-cond-1}
\sigma_j = \Delta J(y_j(t)) = \Delta \rho(y_j(t)) \dot y_j\,,\qquad j=1,\ldots,2N\,.
\end{equation}


%
\begin{figure}
\centering
\begin{tikzpicture} 
				\fill[white] (1,0) -- (3,0) -- (3,3) -- (1,3) -- (1,0);
				
				\draw[thick] (0,-0.5) -- (0,3);
				\draw[thick] (8,-0.5) -- (8,3);
				
				\draw[thick,dashed] (2,0) -- (2,2.5);
				\draw[thick,dashed] (4,0) -- (4,2.5);
				\draw[thick,dashed] (6,0) -- (6,2.5);
				
				\draw[thick] (0,0) -- (8,0);
				\node () at (-0.7,0) {$t=t^{n-1}$};
				
				\draw[thick] (0,2) -- (8,2);
				\node () at (-0.7,2) {$t=t^n$};			
			
		 \draw (-1.2,0.9) node[draw,fill=white,rounded corners] (a) {$y_1(t)$};
\draw[line width=1pt,color=red,-stealth] (a.east) to[bend left] (0.3,0.3);
\draw[line width=1pt,color=red,-stealth] (a.east) to[bend right] (0.3,1.7);
		
		 \draw (4.8,-0.5) node[draw,fill=white,rounded corners] (a) {$y_j(t)$};
\draw[line width=1pt,color=red,-stealth] (a.north) to[bend right] (4.7,0.7);
\draw[line width=1pt,color=red,-stealth] (a.north) to[bend left] (4.3,1.7);
	
				\node () at (0.3,2.5) {\small{$\sigma_1$}};
				
				\node () at (1.7,2.5) {\small{$\sigma_2$}};
		
				\node () at (7.7,0.7) {\small{$\sigma_{2N}$}};
				
				\draw[thick] (0,0) -- (2,2);
				\node () at (0.4,0.7) {\small{$\sigma_1$}};
				
				\draw[thick] (2,0) -- (0,2);
				\draw[thick] (2,0) -- (4,2);
				\node () at (1.6,0.7) {\small{$\sigma_2$}};
				
				\draw[thick] (4,0) -- (2,2);
				\draw[thick] (6,0) -- (4,2);
						
				\draw[thick] (4,0) -- (6,2);
				\draw[thick] (6,0) -- (8,2);
			
				\draw[thick] (8,0) -- (6,2);
				\node () at (7.7,2.5) {\small{$\sigma_{2N}$}};
						
				\draw[thick] (2,2) -- (2.4,2.4);
				
				\draw[thick] (0,2) -- (0.4,2.4);
				\draw[thick] (2,2) -- (1.6,2.4);
				
				\draw[thick] (4,2) -- (3.6,2.4);
				\draw[thick] (4,2) -- (4.4,2.4);
				
				\draw[thick] (6,2) -- (5.6,2.4);
				\draw[thick] (6,2) -- (6.4,2.4);
				
				\draw[thick] (8,2) -- (7.6,2.4);
				
				\draw[<->] (8.2,0) -- (8.2,2);
				\node () at (8.5,1) {\small{$\DT$}};				
\end{tikzpicture}
\caption{Illustration of the polygonals $y_j(t)$ and of the wave strengths $\sigma_j(t)$
}\label{fig:illustration-sigmaj}
\end{figure}
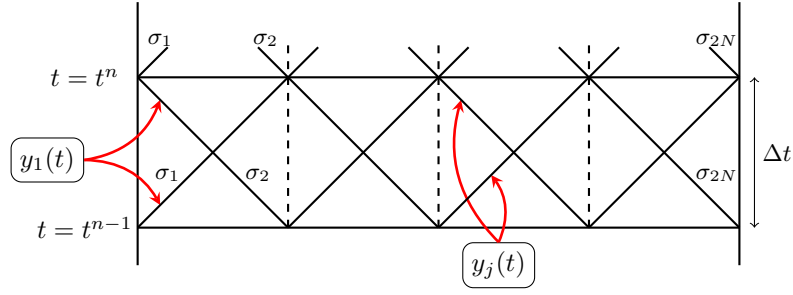			
Now observe that  the function
$$
t\mapsto \int_I \rho^\DX(x,t)\,dx \,;
$$
is continuous and piecewise linear on $\R_+$, and that its derivative is given by
\begin{align}
\frac d{dt}  \int_I \rho^\DX(x,t)\,dx &=  - \sum_{j=1}^{2N} {\Delta \rho}(y_j) \dot y_j \nonumber \\
& =  - \sum_{j=1}^{2N} \Delta J(y_j(t)) = - J(1-,t) + J(0+,t)= 0  \label{eq:sigma-cdot-e}
\end{align}
for every $t\not = t^n$, $t^{n+1/2}$, where we used \eqref{RH-cond-1} and the boundary conditions $J(1-,t) = J(0+,t)=0$, 
which are satisfied exactly for every $t\not= t^n$. Hence \eqref{eq:deriv-int-rho-0} is proved.


Finally, the inequality \eqref{approximate-cons-mass} follows from \eqref{eq:deriv-int-rho-0}, \eqref{int-rho-DX} and 
recalling that $\rho=f^++f^-$. The proof is complete.
\end{proof}

\subsection{Uniform bounds on the Total Variation}\label{subsec:3.3}
We define 
\begin{align}\label{L-pm}
L_{\pm}(t)&=\sum_{(\pm1)-waves} |\Delta f^\pm|\,, \\ \label{L-0}
L_0(t) &= \frac 12 \left(\sum_{0-waves}  |\Delta f^+|+ |\Delta f^-|  \right)
\end{align}
that by \eqref{def-sizes} are related to $\rho$ and $J$ as 
$$
L_{\pm}(t)=  \tv J(\cdot,t)\,, \qquad \qquad L_{\pm}(t) + L_0(t) =  \tv \rho(\cdot,t)\,.
$$
As in the case of the Cauchy problem \cite{AG-MCOM16} and as in \cite{A-A-DS2018}, the functional $L_\pm(t)$ 
may change only at the times $t^n$, due to the interactions with the $(\pm1)-waves$ with the $0-waves$.
Let evaluate the total possible increase of $L_{\pm}$. At each time $t^n$,
by using the inequality \eqref{eq:no-decay}, we get 
\begin{equation*}
L_{\pm}(t^n+) 
\le L_{\pm}(t^n-) + 2 C_0 \left|\bar\alpha_{n} - \bar\alpha_{n-1}\right| \sum_{j=1}^{N-1} \delta_j 
\le L_{\pm}(t^n-) + 2 C_0 \left|\bar\alpha_{n} - \bar\alpha_{n-1}\right| \|k\|_{L^1}\,.
\end{equation*}
Summing up the previous inequality, one gets
\begin{equation}\label{eq:bound-on-Lpm}
L_{\pm}(t^n+) \le L_{\pm}(0+) +  2 C_0 \tv\{\alpha;[0,t_n]\} \|k\|_{L^1}   \,.
\end{equation}
Hence for every $T>0$ the function $[0,T] \ni t \mapsto L_{\pm}(t)$ is uniformly bounded in $t$ and $\DX$.
%
Moreover one has
\begin{align}
L_{\pm}(0+) \le\, & \tv f^+(\cdot,0)  + \tv f^-(\cdot,0) + |J_0(0+)| +  |J_0(1-)|  
+ 2 C_0 \alpha(0+) \|k\|_{L^1} \,,\label{stima-su-Lpm}
\\[2mm]
L_0(t) \le\, & \|\alpha\|_\infty \sum_j |g(J_*(x_j))| \Delta a(x_j) \le C_0 \|\alpha\|_\infty \|k\|_{L^1}  \,.\nonumber
\end{align}
In conclusion,
\begin{align*}
 \tv f^+(\cdot,t) + \tv f^-(\cdot,t) &=\, L_\pm(t) + 2L_0(t) \\[2mm]
&\leq\, \tv f^+(\cdot,0)  + \tv f^-(\cdot,0)   + |J_0(0+)| +  |J_0(1-)|\\[2mm]
&\qquad          + \, 4 C_0 \left( \|\alpha\|_\infty + \tv\{\alpha;[0,T]\} \right)\, \|k\|_{L^1}    
\end{align*}
and hence the total variation of  $t\mapsto (\rho^\DX,J^\DX)(\cdot,t)$ is uniformly bounded on all finite time intervals $[0,T]$, with $T>0$, 
uniformly in $\DX$. 

\subsection{Strong convergence as \texorpdfstring{$\DX\to 0$}{DX to 0} and proof of Theorem~\ref{theorem:well-posedness}}
\label{subsec:convergence}
In this Subsection we prove Theorem~\ref{theorem:well-posedness}, and we start by proving it for $(\rho_0, J_0)\in BV(I)$. 

In this case, for every $T>0$, a standard application of Helly's theorem
implies that there exists a subsequence $(\DX)_j\to 0$ such that ${f^\pm}^{(\DX)_j}\to f^\pm$ in $L^1_{loc} (0,1)\times (0,\infty)$ and that 
$f^\pm: (0,1)\times (0,\infty)\to \R$ are a weak solution to \eqref{GT}.
In terms of ${\rho}^{\DX}$, ${J}^{\DX}$, the identity
\begin{equation}\label{eq:1eq-weak-form}
\int_0^1\int_{0}^{\infty} \left\{ \rho^{\DX} \partial_t \phi + J^{\DX} \partial_x \phi\right\}\,dxdt=0
\end{equation}
holds for every $\phi\in C^1([0,1]\times (0,T))$ (that is, up to the boundaries of $I$) since $J^{\DX}(0+,t)=0 = J^{\DX}(1-,t)$ for every $t\not = t^n$.
Hence the identity \eqref{eq:1eq-weak-form} is satisfied by the strong limit $(\rho,J)$.
Moreover, by passing to the limit as $(\DX)_j\to 0$ in \eqref{approximate-cons-mass} one obtains that \eqref{eq:int-rho-const} holds, that is
\begin{equation*}
\int_I \rho(x,t)\,dx=0\qquad \forall\, t>0\,.
\end{equation*}
To obtain the stability in $L^1$ with respect to the initial data, one can observe that the coupling in system \eqref{GT} 
is \emph{quasimonotone}, in the sense that the equations  
\begin{equation}\label{eq:diag-system-fpm}
\partial_t f^\pm  \pm  \partial_x f^\pm = \mp G\,,\qquad G(x,t,f^\pm) = {k(x) \alpha(t)} \,g(f^+ - f^-)
\end{equation}
satisfy, thanks to the assumptions \eqref{hyp-on-alpha} and \eqref{hyp-weaker-on-k-and-g},
$$
-\frac{\partial G}{\partial f^+} \le 0\,,\qquad
\frac{\partial G}{\partial f^-} \le 0\,.
$$
By adaptation of the arguments in \cite{HN96} (see , which rely on Kru\v{z}kov techniques, one can prove the following stability estimate. For any pair of initial data $(f_0^-, f_0^+)$ and $(\widetilde f_0^-, \widetilde f_0^+) \in L^\infty(I)$, 
let $f^\pm$, $\widetilde f^\pm$ in $(0,1)\times (0,T)$ be solutions of the  problems with the corresponding initial data, according to Definition~\ref{def:weak-sol}. Then the following inequality holds
\begin{equation}\label{ineq:L1-stability}
\|(f^-,f^+)(\cdot,t)  - (\widetilde f^-, \widetilde f^+)(\cdot,t)\|_{L^1(I)} \le \|(f^-_0,f^+_0) - (\widetilde f^-_0, \widetilde f^+_0)\|_{L^1(I)} \,.
\end{equation}
Therefore the weak solution to \eqref{DWE-rho-J-IBVP}--\eqref{init-boundary-data} is unique on $(0,1)\times (0,T)$ 
and can be prolonged for all times, $t\in \R^+$.

Finally, let $(\rho_0, J_0)\in L^\infty(I)$. Then there exists a sequence $\{(\rho_0, J_0)_n\}_{n\in\N}\subset BV(I)$ such that
$(\rho_0, J_0)_n\to (\rho_0, J_0)\in L^1(I)$.  By the $L^1$ stability estimate \eqref{ineq:L1-stability}, the limit in $L^1$ of $f^\pm_n(\cdot,t)$
is well defined and hence also for $(\rho,J)(\cdot,t)$. Since the identity
\begin{equation}\label{eq:1eq-weak-form-BV}
\int_0^1\int_{0}^{\infty} \left\{ \rho_n \partial_t \phi + J_n \partial_x \phi\right\}\,dxdt=0
\end{equation}
holds for every $\phi\in C^1([0,1]\times (0,\infty))$ and for every $n$, then \eqref{eq:1eq-weak-form-BV} is valid also for the strong limit $(\rho,J)$, 
as well as \eqref{eq:int-rho-const}. This completes the proof of Theorem~\ref{theorem:well-posedness}\,.\qed

\begin{rmk} We add more comments about the stability estimate \eqref{ineq:L1-stability}. Due to the quasimonotonicity properties stated above, 
its proof is similar to the one of \cite[Th. 4.1]{HN96}, that was stated for the Cauchy problem of a related system. The presence of the boundary conditions does not provide additional difficulty; let's give a formal argument in support of that. 

From \eqref{eq:diag-system-fpm} and
\begin{equation*}
\partial_t (\widetilde f^\pm)  \pm  \partial_x (\widetilde f^\pm) = \mp {k(x) \alpha(t)} \,g\left(\widetilde f^+ - \widetilde f^-\right)\,,
\end{equation*}
one obtains formally that
\begin{equation*}
    \partial_t \left| f^- - \widetilde f^-\right|  - 
    \partial_x \left| f^- - \widetilde f^-\right| 
    = {k(x) \alpha(t)} \left[g\left(f^+ - f^-\right) - g\left(\widetilde f^+ - \widetilde f^-\right)\right]\cdot \sign\left(f^- - \widetilde f^- \right)\,,
\end{equation*}
as well as
\begin{equation*}
    \partial_t \left| f^+ - \widetilde f^+\right|  +
    \partial_x \left| f^+ - \widetilde f^+\right| 
    = - {k(x) \alpha(t)} \left[g\left(f^+ - f^-\right) - g\left(\widetilde f^+ - \widetilde f^-\right)\right]\cdot \sign\left(f^+ - \widetilde f^+\right)\,.
\end{equation*}
The boundary condition $J(0,t)=0$ translates into 
\begin{equation}\label{bc-for-fpm}
f^+(0,t) = f^-(0,t)\,\qquad     f^+(1,t) = f^-(1,t)
\end{equation}
and similarly 
for $\widetilde f^\pm$. Therefore, after integration in $dx$ over $(0,1)$, the boundary contributions at $x=0$, $x=1$
$$
\left| f^+ - \widetilde f^+\right| - \left| f^- - \widetilde f^-\right| 
$$
vanish while the contribution of the damping term is $\le0$ because of the quasimonotonicity, which relies on the elementary inequality $(a-b)(\sign (a) - \sign (b))\ge 0$ for all $a$, $b\in\R$.

A similar approach was also employed in \cite[Sect.~4.1]{AG-AnIHP16} to provide $L^1$ error estimates for the approximation of the Cauchy problem for \eqref{eq:diag-system-fpm}, in the time-indepedent case.
\end{rmk}

\begin{rmk}
It is possible to introduce the concept of {\rm broad solutions} for the problem \eqref{DWE-rho-J-IBVP}--\eqref{init-boundary-data}, 
by an adaptation of the definition for the Cauchy problem \cite[Sect.3]{Bressan_Book}. 
Indeed, the characteristics can be prolonged for all times by reflection at the boundaries, together with boundary conditions 
\eqref{bc-for-fpm}. The fact that $g$ is only locally Lipschitz continuous in the state variables can be balanced by the presence of the invariant domain, which yields an apriori bound on the solution and hence to the global in time existence of a broad solution.

We expect that the two concepts of solutions coincide in the present setting, that is for $L^\infty$ initial data, especially in view of the uniqueness condition stated in Theorem~\ref{theorem:well-posedness}\,.
\end{rmk}

\section{A finite-dimensional representation of the approximate solutions} 
\label{sec:iter-matrix-discrete}
\setcounter{equation}{0}
In this section we will study the evolution in time of the approximate solution, established in Subsection~\ref{subsec:approximate}, 
by means of a finite-dimensional evolution system of size $2N=2 \DX^{-1}$. 

We remind that the approximate solutions are constructed for the initial--boundary value problem 
\eqref{DWE-rho-J-a}--\eqref{init-boundary-data} with $(\rho_0, J_0)\in BV(I)$ and $\int_I \rho_0(x)\,dx=0$\,.

\subsection{The transition matrix}
Let's introduce a vector representation of the approximate solution that will be the basis of our subsequent analysis.
Define
\begin{equation*}
\T = \{t\ge0: \ t=t^n= n\DT\mbox{ or }t=t^{n+\frac 12}= \left(n+ \frac 12\right)\DT\,,\quad n=0, 1, \ldots \}
\end{equation*}
the set of possible interaction times. At every time $t\not\in \T$, we introduce the vector of the sizes
\begin{equation}\label{def:ssigma}
\ssigma(t)=\left(\sigma_1,\ldots,\sigma_{2N}\right)(t) \in \R^{2N}\,,\qquad N\in2\N
\end{equation}
where, recalling \eqref{def-sizes} and the notation in Proposition~\ref{prop:cons-of-mass}, especially \eqref{def:y-j} and \eqref{RH-cond-1}, one has
\begin{equation}\label{eq:prop-of-sigma-j}
\sigma_j ~\dot = ~\Delta J(y_j) = \Delta \rho(y_j) \dot y_j \,.
\end{equation}
Let's examine its evolution in the following steps.

\begin{itemize}
\item[(1)] At time $t=0+$, $\ssigma(0+)$ is given by the size of the waves that arise at $x_j=j\DX$, with $j=0,\ldots, N$. In particular, 
a (+1) wave arises at $x=0$, two $(\pm1)$ waves arise at each $x_j$ with $j=1,\ldots, N-1$ and finally a  (-1) wave arises at $x=1$.

\medskip

\item[(2)] At every time $t^{n+\frac 12}$, $n\ge0$,
the vector $\ssigma(t)$ evolves by exchanging positions of each pair $\sigma_{2j-1}$, 
$\sigma_{2j}$:
\begin{equation}\label{eq:interaction-B1}
\left(\sigma_{2j-1},\sigma_{2j}\right) \mapsto \left(\sigma_{2j},\sigma_{2j-1}\right)\qquad j=1,\ldots,N
\end{equation}
that results into
\begin{equation}\label{B1}
\ssigma(t+) = B_1 \ssigma(t-)\,,\quad 
B_1\doteq \begin{bmatrix}
0&1&0&\cdots&0&0\\
1&0&0&\cdots&0&0\\
\vdots&\vdots& \ddots & &\vdots &\vdots \\
\vdots&\vdots& &\ddots &\vdots &\vdots \\
0&0&0&\cdots&0&1\\
0&0&0&\cdots&1&0
\end{bmatrix}
\end{equation}

\medskip

\item[(3)] At each time $t^n=n\DT$, $n\ge 1$, the interactions with the Dirac masses at each $x_j$ of the source term occur, 
and we have to take into account the relations introduced in Proposition~\ref{prop:multiple}. 
We will rely on the identity \eqref{mult-inter-matrix-form-gammaPLUS}.%

\smallskip
For each $j=1,\ldots,N-1$, define the \emph{transition coefficients} $\gamma^n_j$ as follows:
\begin{equation}\label{def:c_j}
\gamma^n_j= g'(s_j^n) \delta_j  \bar \alpha_{n}\,,\qquad  j=1,\ldots,N-1\,,\quad n\ge 1, 
\end{equation}
where $\delta_j$ is given in \eqref{delta-j}, that is
$$
\delta_j = \int_{x_{j-1}}^{x_{j}} k(x) dx\,, \qquad j=1,\ldots, N-1\,,
$$
$\bar \alpha_n$ in \eqref{def:alpha_n} and $s_j^n$ satisfies a relation 
as in \eqref{property-of-s}; more precisely
\begin{equation*}
g'(s_j^n) = \frac{g\left(J(x,t^n+) \right)  - g\left(J(x,t^n-) \right) }{J(x,t^n+)  - J(x,t^n-) }\,.
\end{equation*}
Moreover introduce the terms
\begin{align}\label{def:local-sources}
p_{j,n} &= g\left(J(x_{j},t^{n}-) \right) \frac{\delta_j}{1+\gamma^n_j}\,,\qquad j=1,\ldots,N-1\,,\quad n\ge 1 \,.
\end{align}
Then, the local interaction is described as follows:
\begin{align}\label{eq:interaction-B2}
\boxed{
\begin{pmatrix}\sigma_{2j}\\ \sigma_{2j+1}\end{pmatrix} \mapsto 
\frac{1}{1+\gamma^n_j}\begin{pmatrix} \gamma^n_j\sigma_{2j} + \sigma_{2j+1}\\ \sigma_{2j} + \gamma^n_j\sigma_{2j+1} \end{pmatrix}
+  \left( \bar \alpha_n - \bar \alpha_{n-1}\right) p_{j,n} \begin{pmatrix}-1 \\ +1 \end{pmatrix}\,.
}
\end{align}
\medskip
\noindent
To recast it in a global matrix form, we define
\begin{equation}\label{def:cc}
\gg^n= \left(\gamma^n_1,\ldots,\gamma^n_{N-1}\right)\in\R^{N-1}  
\end{equation}
and set 
\setlength{\fboxsep}{5pt}
\begin{equation} \label{B2}
B_2(\gg^n) =  \begin{bmatrix}
   ~ 1                                    \\[1mm]
      &  \framebox[25pt][c]{$\hat A^n_{1}$}         &   & \text{\LARGE 0}
      \\
      &               & \ddots               \\
      & \text{\LARGE0} &   &  \framebox[30pt][c]{$ \hat A^n_{N-1} $}        \\
      &               &   &   & 1~
    \end{bmatrix}\,,\qquad   
\hat A^n_{j} = \frac{1}{1+\gamma^n_j}\begin{bmatrix} \gamma^n_j & 1\\
1& \gamma^n_j 
\end{bmatrix}\,.
\end{equation}
\smallskip\par\noindent
The matrix $B_2(\gg)$ is tridiagonal 
with diagonal components as follows,
\begin{equation*}
\left(1, \frac{\gamma^n_1}{1+\gamma^n_1}, \frac{\gamma^n_1}{1+\gamma^n_1}, \frac{\gamma^n_2}{1+\gamma^n_2},\ldots, 
\frac{\gamma^n_{N-2}}{1+\gamma^n_{N-2}},\frac{\gamma^n_{N-1}}{1+\gamma^n_{N-1}}, \frac{\gamma^n_{N-1}}{1+\gamma^n_{N-1}},1\right) \in \R^{2N}
\end{equation*}
and subdiagonals
\begin{equation*}
\left(0, \frac{1}{1+\gamma^n_1}, 0,\frac{1}{1+\gamma^n_2},0,
\ldots,\frac{1}{1+\gamma^n_{N-1}},0\right) \in \R^{2N-1}\,.
\end{equation*}
Hence $\ssigma(t)$ evolves according to
\begin{align*}
\ssigma(t^n+) = B_2(\gg^n) \ssigma(t^n-) +  \left( \bar \alpha_n - \bar \alpha_{n-1}\right)  \GG_n
\end{align*}
with 
\begin{align} \label{def:G}
\GG_n &=  \left(0, -p_{1,n},  + p_{1,n}, \ldots, -p_{N-1,n}, + p_{N-1,n}, 0\right)^t  \,.
\end{align}
\end{itemize}
We summarize the previous identities to get the following statement.
\begin{proposition}
At time $t=t^n$ let $B_1$, $B_2(\gg^n)$, $\GG_n$ be defined by \eqref{B1}, \eqref{B2},  \eqref{def:G} respectively. Define
\begin{equation}\label{BB}
{B}(\gg^n) :=  B_2(\gg^n)  B_1\,.
\end{equation} 
Then the following relation holds,
\begin{equation}\label{def:iteration}
\boxed{ \ssigma(t^n+) = B(\gg^n) \ssigma(t^{n-1}+) +  \left( \bar \alpha_n - \bar \alpha_{n-1}\right)  \GG_n }\,, \qquad n\ge 1\,.
\end{equation}
\end{proposition}

\bigskip
\begin{rmk} We give a couple of remarks about the use of the local interaction estimates \eqref{mult-inter-matrix-form}, 
\eqref{mult-inter-matrix-form-gammaPLUS}.

\begin{enumerate}[(a)]
\item If, in place of  \eqref{mult-inter-matrix-form-gammaPLUS}, the relation  \eqref{mult-inter-matrix-form} is used, 
the quantities \eqref{def:c_j} and \eqref{def:local-sources} are defined by
$$
\gamma^n_j= g'(s_j^n) \delta_j  \bar \alpha_{n-1}\,,\qquad p_{j,n} = g\left(J(x_{j},t^{n}+) \right) \frac{\delta_j}{1+\gamma^n_j}\,.
$$

\item Notice that, in \eqref{eq:interaction-B2}, we consider the \emph{space order} instead of the \emph{family order},
that was used in \eqref{mult-inter-matrix-form}. That is,
$$
(\sigma_{2j},\sigma_{2j+1}) = \begin{cases} (\sigma_{1}^-,\sigma_{-1}^-) & \mbox{ before the interaction }\\[2mm]
(\sigma_{-1}^+,\sigma_{1}^+) & \mbox{ after the interaction}\,.
\end{cases}
$$
\end{enumerate}
\end{rmk}
		
\subsection{Properties of the transition matrix}	
As observed in \cite{A-A-DS2018}, the matrix $B$ in \eqref{BB} is doubly stochastic (that is, it is non-negative and the sum of all the elements by row is $1$, as well as by column) for every vector $\gg$; 
we will call it \emph{transition matrix}. Notice that it is non-negative provided that all the $\gamma^n_j$ (see \eqref{def:c_j})
are non-negative, which relies on the assumption that $g'\ge0$.  Let's summarize some properties: 

\begin{itemize}
\item[(i)] its eigenvalues $\lambda_j$ satisfy $|\lambda_j|\le 1$ for all $j=1,\ldots,2N$;

\item[(ii)] if $\gamma_j \cdot \gamma_{j+1}>0$ for some $j$,  
then the eigenvalues with maximum modulus are exactly two ($\lambda=\pm 1$) and they are simple. 

\item[(iii)] The values $\lambda=\pm 1$ are eigenvalues with corresponding (left and right) eigenvectors
\begin{align}\label{v-pm}
\begin{aligned}
\lambda_-= -1\,,\qquad & v_{-}=(1,-1,-1,1,\ldots,1,-1,-1,1)\,,\\
\lambda_+= 1\,,\qquad & e=(1,1,\ldots,1,1)\,.
\end{aligned}
\end{align}
\end{itemize}
Moreover $B(\zero)$ is a normal matrix, since it is a permutation and hence $B(\zero)^t B(\zero)=B(\zero) B(\zero)^t = \Id$.
This property does not hold if $\gg\not =\zero$.

\begin{rmk}\label{rem:properties-of-B}
The properties established in Subsections~\ref{subsec:3.2}--\ref{subsec:3.3} can be rewritten in terms of the vectorial representation 
of the solution \eqref{def:ssigma}, as follows\,. 

\begin{enumerate}[(a)]
\item \emph{(Boundary conditions)}\quad  From equation \eqref{eq:sigma-cdot-e} it follows that 
\begin{equation}\label{sigma-dot-e-is-zero}
\ssigma(t) \cdot e=0 
\end{equation}
for every $t\not \in \T$. 
Indeed, 
\begin{equation*}
 \ssigma(t) \cdot e= \sum_{j=1}^{2N} \sigma_j(t) = \sum_{j=1}^{2N} \Delta J(y_j(t)) = J(1-,t) - J(0+,t)= 0 \,. 
\end{equation*}

\item \emph{(Total variation)}\quad The quantity $L_{\pm}(t)$ coincides with $\|\ssigma (t)\|_{\ell_1}$. In particular, from
\eqref{eq:bound-on-Lpm}--\eqref{stima-su-Lpm} we obtain
\begin{align}
\|\ssigma (0+)\|_{\ell_1} &\le  \tv f^+(\cdot,0) + \tv f^-(\cdot,0) + |J_0(0+)| +  |J_0(1-)|  
+ 2 C_0 \alpha(0+) \|k\|_{L^1} \,, \label{stima-su-sigma-0-ell1}
\\[2mm] \nonumber
\|\ssigma (t)\|_{\ell_1}  &\le  \|\ssigma (0+)\|_{\ell_1} +  2 C_0 \tv\{\alpha;[0,t_n]\} \|k\|_{L^1},\qquad t^n<t< t^{n+1}\,.
\end{align}


\item The following property holds, 
\begin{equation}\label{eq:sigma-dot-v_-}
\left|\ssigma(t)\cdot v_{-}\right| \le \left| \ssigma(0+)\cdot v_- \right|  \le  \tv \{\bar J_0;[0,1]\} 
 \qquad \forall\, t\not\in\T
\end{equation}
where $v_-$ is the eigenvector corresponding to $\lambda=-1$, see \eqref{v-pm}, and 
\begin{equation*}
\bar J_0(x) = \begin{cases}
J_0(x) & x\in (0,1)\\
0 & x\in 0 \ \mbox{\rm  or }1\,.
\end{cases}
\end{equation*}
Indeed, the second inequality in \eqref{eq:sigma-dot-v_-} follows from \cite[(77)]{A-A-DS2018}. 
To prove the first inequality in \eqref{eq:sigma-dot-v_-}, we first consider $t\in(t^n,t^{n+1/2})$ and use the iteration formula \eqref{def:iteration}
to obtain
\begin{align*}
\ssigma(t)\cdot v_{-} = \ssigma(t^n)\cdot v_{-}=B(\gg^n) \ssigma(t^{n-1}+)\cdot v_- + \GG_n\cdot v_{-}\,.
\end{align*}
By recalling the definition of \eqref{def:G}, we immediately deduce that  
\begin{equation*}
\GG_n\cdot v_{-}=0\qquad \forall\, n\,,
\end{equation*}
and therefore that
\begin{align*}
\ssigma(t)\cdot v_{-} 
&= \ssigma(t^{n-1}+)\cdot B(\gg^n)^t v_- \\
& = - \ssigma(t^{n-1}+)\cdot v_-\\
& = (-1)^n \ssigma(0+)\cdot v_-\,,
\end{align*}
from which \eqref{eq:sigma-dot-v_-} follows for $t\in(t^n,t^{n+1/2})$. Secondly, for $t\in(t^{n+1/2},t^{n+1})$, by using \eqref{B1} we have that
$$
\ssigma(t) = \ssigma(t^{n+1/2}+) = B_1  \ssigma(t^{n+1/2}-) = B_1  \ssigma(t^{n}+)\,,\qquad  t\in(t^{n+1/2},t^{n+1})
$$
and hence
\begin{align*}
\ssigma(t)\cdot v_{-} =  \ssigma(t^{n}+)\cdot B_1 v_{-} = - \ssigma(t^{n}+)\cdot v_{-}
\end{align*}
from which it follows again \eqref{eq:sigma-dot-v_-}\,.

\item The \emph{undamped} equation: $k(x)\equiv 0$.  

In this case, each vector $\GG_n$ vanishes and $\gg^n=\zero$. Therefore from \eqref{def:iteration} and \eqref{eq:interaction-B1} we obtain
\begin{equation*}
\ssigma(t) = \begin{cases} 
B(\zero)^n \ssigma(0+) & t^n<t<t^{n+\frac 12}\\
B_1B(\zero)^n \ssigma(0+) & t^{n+\frac 12}<t<t^{n+1}\,.
\end{cases}
\end{equation*}
Since every wave-front issued at $t=0$ reflects on the two boundaries and gets back to the initial position after a time 
$T=2 = 2N\DT$, it is clear that 
\begin{equation}\label{B-zero-2N}
B(\zero)^{2N}= \Id
\end{equation} 
that is, $B(\zero)^{2N}$ coincides with the identity matrix in $M_{2N}$. As a consequence, the powers of $B(\zero)$ are periodic with period $2N$:
\begin{equation*}
B(\zero)^{n+2N} = B(\zero)^{n}\,,\qquad n\in \Z\,. 
\end{equation*}
With a similar argument one can prove that
\begin{equation}\label{B-zero-N}
(B(\zero)^{N})_{ij}= \begin{cases}1 & \mbox{ if }~ i+j=2N+1\\
0 & \mbox{ otherwise,} \end{cases}
\end{equation} 
that is, $B(\zero)^{N}$ is the matrix with component 1 on the antidiagonal positions $(i, 2N+1-i)$ and 0 otherwise. It is clear that
$(B(\zero)^{N})^2 = B(\zero)^{2N}=\Id$.
\end{enumerate}
\end{rmk}

\subsection{A representation formula for \texorpdfstring{$\rho$ and $J$}{rho and J}}
In this subsection we provide a pointwise representation of $\rho(x,t)$, $J(x,t)$ by means of the vectorial quantity $\ssigma(t)$.
It is based on the key properties \eqref{eq:prop-of-sigma-j} and \eqref{J*_rho*}$_{2}$, that we recall here for convenience: 
for $y_j$ given in \eqref{def:y-j}, 
\begin{equation}\label{eq:remind-disc}
\begin{cases}
\sigma_j=\Delta J(y_j) = \Delta \rho(y_j) \dot y_j & x=y_j(t)\,,\\
\Delta \rho(x_j) = -2\alpha(t) g(J(x_j))\delta_j\,, \quad \Delta J(x_j)=0 &  x=x_j=j\DX
\end{cases}\qquad j=1,\ldots,2N
\end{equation}
Therefore we can reconstruct the functions $x\to \rho(x,t)$ and $x\to J(x,t)$ as stated in the following Proposition.
We define
\begin{equation}\label{def:v-ell}
\vv_{0}=\zero_{2N}\,, \qquad \vv_{\ell}=(\underbrace{1,\cdots,1}_{\ell},0,\cdots,0) \in\R^{2N},\quad  \ell=1,\cdots,2N
\end{equation}
and
\begin{align}
H=\left\{\vv_{\ell}\in\R^{2N}, \quad  \ell=0,\cdots,2N\right\}
\,.\label{set-of-vv}
\end{align}  

\begin{lemma}(\textbf{Representation formula for $\rho$, $J$, $f^\pm$})\label{prop:representation-J-rho}
\par\noindent
For every $(x,t)$ with $x\not = y_j(t)$ and $t\in(t^n,t^{n+1})$, the following holds.
\begin{enumerate}
\item 
There exists $\vv=\vv(x)\in H$ such that
\begin{equation}\label{def:J(x,t)}
J(x,t)  = \ssigma(t)\cdot \vv(x)\,.
\end{equation}
In particular
\begin{equation}\label{rmk:vv}
\vv(x_j)=\vv_{2j}\,,\qquad  j=0,\ldots,N\,.
\end{equation}

\item If moreover $x\not=x_j$, then the following holds:
\begin{equation}
\label{def:rho(x,t)}
\rho(x,t) =  \widetilde\ssigma(t)\cdot  \vv(x) + \rho(0+,t) - 2 \bar \alpha_n\sum_{j:\ x_j < x} g(J(x_j,t)) \delta_j\,, 
\end{equation}
where $ \bar \alpha_n$ is defined in \eqref{def:alpha_n}, 
\begin{equation}\label{def:tilde-ssigma}
\widetilde \ssigma(t)= \pm \Pi \ssigma (t) = \begin{cases} \Pi \ssigma& t\in\left(t^n, t^{n+1/2}\right)\\
- \Pi \ssigma(t) & t\in\left(t^{n+1/2}, t^{n+1}\right)
\end{cases}
\end{equation}
and
\begin{equation}\label{def:Pi}
\Pi={\rm diag} (1,-1,1,-1,\ldots,1,-1)\in M_{2N}\,.
\end{equation}

\smallskip
\item Finally, for  $j=0,\ldots,N-1$ one has that
\begin{equation}\label{eq:representation-fpm}
f^\pm(x_j+,t) = \ssigma(t)\cdot \vv^\pm_{2j} 
+  \frac 12 \rho(0+,t)  
-  \bar \alpha_n \sum_{0\le\ell \le j} g(J(x_\ell,t)) \delta_\ell   
\end{equation}
where
\begin{equation}\label{def:vv-pm}
\begin{aligned}
\vv^+_{2j} &=  \frac 1{2}  \left( \Pi + \Id\right) \vv_{2j} = (\underbrace{1,0,\ldots,1,0}_{2j},0,0,\ldots,0,0)\\
\vv^-_{2j}  &=   \frac 1{2}  \left( \Pi - \Id\right) \vv_{2j} = - (\underbrace{0,1,\ldots,0,1}_{2j},0,0,\ldots,0,0)\,.
\end{aligned}
\end{equation}
\end{enumerate}
\end{lemma}

\begin{proof} {\it (1)}\quad 
About \eqref{def:J(x,t)}, it is enough to observe that
\begin{equation*}
J(x,t) = \underbrace{J(0+,t)}_{=0}   + \sum_{y_\ell(t) < x} \Delta J(y_\ell) = \sum_{y_\ell < x} \sigma_\ell(t)\,. 
\end{equation*}
Hence 
$$J(x,t)  =  \ssigma(t)\cdot \vv_{\bar \ell}
$$ 
with $\bar \ell\in \{0,1,\ldots,2N-1\}$ such that
\begin{equation}\label{def:bar-ell}
y_{\bar \ell}<x< y_{\bar \ell+1}\,.
\end{equation}
In particular, if $x_j = j\DX$, then 
\begin{equation*}
J(x_j,t) = \underbrace{J(0+,t)}_{=0}   + \sum_{y_\ell(t) < x_j} \Delta J(y_\ell) = \sum_{\ell=1}^{2j} \sigma_\ell(t) = \ssigma(t)\cdot \vv_{2j}\,.
\end{equation*}
Hence \eqref{rmk:vv} is proved.

\smallskip\noindent
  {\it (2)}\quad To prove \eqref{def:rho(x,t)}, let's write $\rho(x,t)$ for $x\not=x_j$ and $x\not = y_\ell$ as follows:
\begin{align*}
\rho(x,t) &= \rho(0+,t)   + \underbrace{\sum_{y_\ell < x} \Delta \rho(y_\ell,t)}_{(a)} +  \underbrace{\sum_{x_j < x} \Delta \rho(x_j,t)}_{(b)} \,.
\end{align*}
Indeed, differently from $J$, the component $\rho$ varies also along the 0-waves.
About $(a)$, by recalling the first relation in \eqref{eq:remind-disc}, we get 
\begin{align*}
\sum_{y_\ell < x} \Delta \rho(y_\ell,t)=\sum_{y_\ell < x} \sigma_{\ell}\, \dot y_\ell \,.
\end{align*}
Now, notice that (see Figure~\ref{fig:illustration-sigmaj})
$$
\dot y_j(t) = \begin{cases} 1& j \mbox{ odd}\\
-1 &  j \mbox{ even}
\end{cases}\qquad t\in\left(t^n, t^n+\frac{\DT}2\right)
$$
as well as
$$
\dot y_j(t) = \begin{cases} -1& j \mbox{ odd}\\
1 &  j \mbox{ even}
\end{cases}\qquad t\in\left(t^n+\frac{\DT}2, t^{n+1}\right)\,.
$$ 
Therefore $(a)$ is of the form
$$
\sum_{y_\ell < x} \Delta \rho(y_\ell,t) =   \widetilde  \ssigma(t)\cdot \vv_{\bar \ell}\,.
$$

\smallskip\noindent
Concerning $(b)$, since $\Delta \rho(x_j) = -2 g(J(x_j))\delta_j$ we immediately get
$$
\sum_{x_j < x} \Delta \rho(x_j,t)= - 2 \bar \alpha_n\sum_{x_j < x} g(J(x_j,t)) \delta_j\,.
$$
Therefore the proof of \eqref{def:rho(x,t)} is complete.

{\it (3)} \quad Finally, about \eqref{eq:representation-fpm}, we use the relation $f^\pm=\frac{\rho\pm J}2$ to get
\begin{equation*}
f^\pm(x_j+,t) =  \frac{\widetilde\ssigma(t) \pm \ssigma(t)}2 \cdot  \vv(x_j) +  \frac 12 \rho(0+,t)  
-  \bar \alpha_n \sum_{0\le\ell \le j} g(J(x_\ell,t)) \delta_\ell\,.
\end{equation*}
We rewrite the first term as follows,
\begin{align*}
\frac{\widetilde\ssigma(t) \pm \ssigma(t)}2 \cdot  \vv(x_j)  &= \frac12 \left(\Pi\pm I_{2N}\right) \ssigma(t)\cdot  \vv(x_j) \\
&=  \ssigma(t)\cdot \underbrace{\frac12 \left(\Pi\pm I_{2N}\right) \vv_{2j}}_{= \vv^\pm_{2j}}
\end{align*}
where we used \eqref{rmk:vv} and the fact that the matrices $\Pi\pm I_{2N}$, 
\begin{align*}
\frac 1{2}  \left( \Pi + \Id\right) &= {\rm diag} (1,0,1,0,\ldots,1,0)\,,\\
\frac 1{2}  \left( \Pi - \Id\right) &= - {\rm diag} (0,1,0,1,\ldots,0,1)
\end{align*}
are symmetric. The proof of \eqref{eq:representation-fpm} is complete.
\end{proof}

\begin{rmk}\label{rmk:vv-2} Here is a list of remarks about the representation formulas in Lemma~\ref{prop:representation-J-rho}.

\begin{enumerate}[(a)]

\item The value of $\rho(0+,t)$ in \eqref{def:rho(x,t)} is determined by the conservation of mass identity: 
$$\int_I \rho^\DX(x,t)\,dx =  \int_I \rho^\DX(x,0)\,dx\,.$$

\medskip
\item By the definitions \eqref{def:vv-pm}, \eqref{B1} of $\vv^+_{2j}$ and $B_1$, respectively, it is immediate to find that
\begin{equation}\label{eq:id-B1-vpm}
B_1 \vv^\pm_{2j}  = - \vv^\mp_{2j} \,.
\end{equation}

\item The last term in \eqref{eq:representation-fpm}, which is related to the variation of $f^\pm$ across the point sources $x_j$, 
can be also conveniently expressed as a scalar product with $\vv^\pm_{2j}$. Indeed, if we define
\begin{align*}
\widehat p_{j}(t) &= g(J(x_j,t)) \delta_j \\
\widehat{\GG}(t) &= \left( 0, - \widehat p_{1}, \widehat p_{1}, \ldots, -\widehat p_{N-1},  \widehat p_{N-1},0 \right)^t 
\end{align*}
\end{enumerate}
then it is immediate to verify the following identity holds:
\begin{equation}\label{vector-source-term}
\sum_{0\le\ell \le j} g(J(x_\ell,t)) \delta_\ell = \widehat{\GG}(t) \cdot \vv^-_{2j} = \widehat{\GG}(t) \cdot \vv^+_{2j+2} \,.
\end{equation}
Notice the similarity between $\widehat{\GG}$, for time $t=t^n-$, and the \emph{vector source term} 
$\GG_n$ defined at \eqref{def:G}. In general, the map $t\mapsto \widehat{\GG}(t)$ is nonlinear with respect to $\ssigma(t)$ 
because of the nonlinearity of $J\mapsto g(J)$. In the following section, we will analyze in detail the case of $g$ being linear.
\end{rmk}

\section{The linear case: the telegrapher's equation}\label{sec:linear-case}
\setcounter{equation}{0}

In this section we assume that, for some $d>0$,
\begin{equation*}
k(x)\equiv d
\,,\qquad  g'(J)\equiv 1
\,,\qquad \alpha(t)\equiv  1 
\end{equation*} 
which corresponds to the case of the standard telegrapher's equation:
\begin{equation}
\begin{cases}
\partial_t\rho +  \partial_x J  = 0, &\\
\partial_t J  +  \partial_x \rho = - 2 d J\,. &
\end{cases} \label{eq:telegrapher}
\end{equation}
Let's summarize the results of Section~\ref{sec:iter-matrix-discrete} in the present context.

\smallskip
$\bullet$ The vector $\gg$, defined at \eqref{def:cc}, has all equal components:
\begin{equation}
\begin{aligned}
&\gamma = d\DX = \frac {d}N\,, \\[1mm]
&  \gg=\gamma (1,\ldots,1)\,.
\end{aligned} \label{eq:gamma-costant}
\end{equation}
Hence the iteration formula \eqref{def:iteration} leads to 
\begin{equation}\label{def:iteration-d}
\ssigma(t^{n}+)=B(\gg)^{n}\ssigma(0+)\,.
\end{equation}
For $d=0$ and hence $\gg=\zero$, it is clear that the sequence in \eqref{def:iteration-d} corresponds to the undamped linear system
\begin{equation*}
\partial_t\rho +  \partial_x J  = 0 = \partial_t J  +  \partial_x \rho\,,
\end{equation*}
see {$(d)$} in Remark~\ref{rem:properties-of-B}.  

\smallskip
$\bullet$ The representation formula \eqref{eq:representation-fpm} for $x=x_j\pm$, here, reads as:
\begin{equation}\label{eq:represent-fpm-linear}
\begin{aligned}
f^\pm(x_j+,t) &= \ssigma(t) \cdot  \vv^\pm_{2j} +   \frac 12 \rho(0+,t) 
-   \frac d N \sum_{0\le \ell \le j } J(x_\ell,t) \,,\qquad j=0,\ldots,N-1\\
f^\pm(x_j-,t) &= \ssigma(t) \cdot  \vv^\pm_{2j} +   \frac 12 \rho(0+,t) 
-   \frac d N \sum_{0\le \ell < j } J(x_\ell,t) \,,\qquad j=1,\ldots,N
\end{aligned}
\end{equation}
where $x_j = j\DX = \frac jN$ and $\vv^\pm_{2j}$  are defined at \eqref{def:vv-pm}.

The plan of this section is the following. First we set the ground to study the long time behavior of \eqref{def:iteration-d}, 
through the expansion formula established in Theorem~\ref{theo:exp-formula}, Subsection~\ref{subsec:exp-formula}.  
Then we prove two contractivity properties for \eqref{def:iteration-d}:

- in Subsection~\ref{subsec:sum-norm} we analyze the matrix norm induced the $\ell_1$--norm and improve a statement already given in \cite{A-A-DS2018}; 

- while in Subsection~\ref{subsec:contractivity-inv-dom}
 we address the contractivity of the invariant domain $[m,M]$ for the state variables $f\pm$, stated in Theorem~\ref{th:Linfty-d}.

\smallskip We remark that the contractivity property established in Subsection~\ref{subsec:contractivity-inv-dom} would yield a decay property for the $BV$ norm of the solution, 
as obtained in \cite{A-A-DS2018}; however this would not be sufficient to reach an analogous property for the $L^\infty$ norm. This is our main motivation in pursuing the result 
of Theorem~\ref{th:Linfty-d}.   

\subsection{An expansion formula}\label{subsec:exp-formula} 
In this subsection we provide an expansion formula for \eqref{def:iteration-d} 
for the power $n=N$, that corresponds to the time $t=1$. The expansion is made in terms of the parameter 
$\gamma = \frac {d}N$, with $d>0$ and $N\to\infty$.

With $\gg$ as in \eqref{eq:gamma-costant}, the matrix $B(\gg)$ can be decomposed as the convex combination of two matrices, see also \cite[p.  185, Proposition 5]{A-A-DS2018}:
\begin{equation}\label{decompose-B}
		B(\gg)=\frac{1}{1+\gamma}\left( B(\zero)+ \gamma B_1\right)\,.
\end{equation}
Thanks to this decomposition, we can analyze the powers of $B(\gg)$. For a generic $n\in \N$ one has 
\begin{align} \label{eq:expansion}
B(\gg)^{n} & =(1+\gamma)^{-n}\left[B(\zero)+ \gamma B_1 \right]^{n} \,,\qquad n\ge 1\,.
\end{align}
The factor $(1+\gamma)^{-n}$ provides an exponentially decreasing term with respect to time. 
Indeed let $T>0$ and recalling that $\DT=N^{-1}$, we have
\begin{equation}\label{eq:time-convergence}
\left(1+ \frac {d}N \right)^{-[TN]} \to \ee^{-dT}\qquad N\to\infty \,.
\end{equation}
Let us focus on the second factor in \eqref{eq:expansion}, that is $\left[B(\zero)+ \gamma B_1 \right]^{n}$. In \cite[Theorem 10]{A-A-DS2018}
an expansion formula is provided in terms of $d$ and $N$  for the power $n=2N$. The following theorem 
states a similar expansion for the power $n=N$, 
which turns out to be a more convenient choice.
\begin{theorem}\label{theo:exp-formula}
Let $N\in 2\N$ and $d\ge 0$. Then the following identity holds
\begin{align}\label{eq:n_0-N}
\left[B(\zero)+ \frac d N B_1 \right]^{N} &=  B(\zero)^{N} + d\widehat{P} + R_{N}(d)
\end{align} 
where 
\begin{align} \label{def:hat_P}
\widehat P  & = \frac {1}{2N} \left(e^t e  + v_-^t v_-  \right)\,,\\[2mm]
\label{eq:R_n}
R_{N}(d) & = \sum_{j=0}^{N-1}   \zeta_{j,N} B_1 B(\zero)^{N-2j-1} 
 + \sum_{j=1}^{N-1}  \eta_{j,N} B(\zero)^{2j-N} \,.
\end{align} 
The coefficients $\zeta_{j,N}$ and $\eta_{j,N}$ depend on $d$ and satisfy the following estimate:
\begin{align}\label{stima-su-zeta-eta_jN}
0\le \sum_{j=0}^{N} \zeta_{j,N}  + \sum_{j=1}^{N} \eta_{j,N} 
&\le \ee^{d}  - d -1+ {\frac {K} N} 
\end{align}
where $K=K(d)\ge 0$ is independent on $N$, and $K(d) \to0$ as $d\to 0$.
\end{theorem}
The proof is deferred to Appendix~\ref{appendix:proof-of-exponential-formula}. For the definition of $K=K(d)$ see \eqref{K-sum-modif-bessel-0}\,.

\smallskip
In the following, the analysis will be based on the equation \eqref{def:iteration-d} for $n=N$. Notice that $t^N=N\DT =1$. By recalling \eqref{eq:expansion} and the expansion formula \eqref{eq:n_0-N}, we get 
\begin{equation} \label{eq:sigma-thN}
\boxed{\begin{aligned}
\ssigma(t^{N}+) &=B(\gg)^{N}\ssigma(0+)\\  
&=  \left(1+\frac d N \right)^{-N} \left(B(\zero)^{N} + d \widehat{P} + R_{N}(d)\right)\ssigma(0+)\,.
\end{aligned}}
\end{equation}
Recalling \eqref{def:tilde-ssigma}, one obtains a similar expression for 
\begin{equation} \label{eq:sigma-thN-tilde}
\widetilde\ssigma(t^{N}+)=\Pi B(\gg)^{N} \ssigma(0+)\,.
\end{equation}
We remind that $\ssigma$ is used in the representation formula for $J$, while $\widetilde \ssigma$ is used in the one for $\rho$.

In the formula \eqref{eq:sigma-thN}, an expansion in powers of $d$ is obtained, since $R_{N}(d)$ can be expressed in terms of powers $d^\ell$ with $\ell\ge 2$. A key point is the identification of the first order term $\widehat{P}$, that will lead us to a cancellation property stated 
in the following proposition. 

\begin{proposition} The following identity holds,
\begin{equation}\label{eq:hatPww-ssigma}
\widehat{P}\ssigma(0+) = \frac{1}{2N}\big(\ssigma(0+)\cdot v_{-}\big) v_{-}\,.
\end{equation}
\end{proposition}

\begin{proof} By recalling the definition of $\widehat{P}$ in \eqref{def:hat_P}, one has that
\begin{equation}\label{eq:hatPww}
\widehat{P}\ww = \frac{1}{2N}\big(\left(\ww\cdot e\right)e+\left(\ww\cdot v_{-}\right)v_{-}\big)\qquad \forall\, \ww\in \R^{2N}\,.
\end{equation}
By setting $\ww=\ssigma(0+)$, from \eqref{sigma-dot-e-is-zero} we immediately get \eqref{eq:hatPww-ssigma}.
\end{proof}

\subsection{Contractivity of the "sum" norm}\label{subsec:sum-norm}
Next, for a fixed $T>0$, we seek an estimate on $B(\gg)^n$ as $n=[NT]$ and $N\to\infty$.
In \cite[Proposition 11 and (88)]{A-A-DS2018}, it is proved that the matrix norm induced by $\|\cdot\|_{\ell_1}$ (also called \emph{sum} norm, \cite{Horn-Johnson}) is contractive for $B(\gg)^n$ on the subspace 
\begin{equation}\label{def:E-}
{E_-} \, \dot = \, <e,v_->^\perp
\end{equation}
which is the linear space generated by all the eigenvectors of those eigenvalues $\lambda$ such that $|\lambda|<1$. Here we provide an extension 
of this property, that leads to an estimate for the time $T=1$.

\begin{proposition}\label{prop:ell-1norm}
Let $N\in2\N$ and $d\ge 0$. There exists a constant $C_N(d)$ (see \eqref{eq:Ctilde-hN} below) such that

\smallskip\par\noindent
\begin{equation}\label{def:Chd}
C_N(d)\to (1-d \ee^{-d})\,\dot{=}\,C(d) <1\,,\qquad N\to\infty 
\end{equation}
\smallskip\par\noindent
and that, for all $\ww \in \R^{2N}$,
\begin{equation}\label{eq:norm1}
\bigl\| B(\gg)^{N}\ww \bigr\|_{\ell_1} \le C_N(d) \bigl\|\ww \bigr\|_{\ell_1} + d \left(1+\frac d N\right)^{-N} \left(|\ww\cdot e| + |\ww\cdot v_-| \right)\,.
\end{equation}
In particular, for $N$ large enough such that $C_N(d)<1$, the $\ell_1$--norm is contractive on the subspace $E_-$ defined at \eqref{def:E-}.
\end{proposition}
\begin{proof}
Let $\ww\in \R^{2N}$. By means of the formula \eqref{eq:expansion} 
and the expansion formula \eqref{eq:n_0-N}, we  obtain
\begin{align*}
B(\gg)^{N}\ww & =\left(1+\frac dN\right)^{-N}\left[B(\zero)+ \frac dN B_1 \right]^{N}\ww\\
&= \left(1+\frac dN\right)^{-N} \left[ B(\zero)^{N}\ww +  \frac{d}{2N} \big(\left(\ww\cdot e\right)e+\left(\ww\cdot v_{-}\right)v_{-}\big)+ R_{N}(d)\ww \right]
\end{align*}
where we used \eqref{eq:hatPww}.

Let $||\cdot||$ be a vector norm that is invariant under components permutation of the vectors. Since $B(\zero)^{N}$ 
is permutation matrix and $R_{N}(d)$ is a linear combination of permutation matrices, we use \eqref{stima-su-zeta-eta_jN} to get that
\begin{align*}
||B(\gg)^{N}\ww|| & \le  \left(1+\frac dN\right)^{-N} ||\ww|| \left(1 +  \ee^{d}  - d -1 + \frac KN \right) \\
&\qquad +   \left(1+\frac dN\right)^{-N}  \frac{d}{2N} \left( |\ww\cdot e| \cdot ||e|| +|\ww\cdot v_{-}| \cdot ||v_-||\right)\,.
\end{align*}
In particular, the above estimate holds for $$||\cdot||=\|\cdot\|_{\ell_1}\,.$$ 
Since $\|e\|_{\ell_1}= \|v_-\|_{\ell_1}=2N$, if we set 
\begin{equation}\label{eq:Ctilde-hN}
C_{N}(d) \, \dot = \, \left(1+ \frac {d}N \right)^{-N} \left[\ee^{d}-d ~+~{\frac1N} K(d)\right]
\end{equation}
then the estimate \eqref{eq:norm1} follows. The proof of Proposition~\ref{prop:ell-1norm} is complete.
\end{proof}

The formula \eqref{eq:norm1} indicates that, as $N\to\infty$, 
\begin{align}\label{ineq:ell-1-contraction}
\bigl\| B(\gg)^{N}\ww \bigr\|_{\ell_1} &\le C_N(d) \bigl\|\ww \bigr\|_{\ell_1}\qquad \ww\in E_-\,,\\[1mm] \nonumber
\lim_{N\to\infty} C_N(d)  & = C(d) <1\,.
\end{align}
This implies that the matrix norm induced by the $\ell_1$--norm is asymptotically contractive 
for the power $B(\gg)^{N}$ on the subspace $E_-$, the norm being defined by
\begin{equation*}
    \bigl\|| B(\gg)^{N} \bigr\||_{1} = \max_{\|\ww\|_{\ell_1} =1} \bigl\| B(\gg)^{N}\ww \bigr\|_{\ell_1}\,.
\end{equation*}

Of course, for $\gamma=d/N$ and $N$ fixed, the sequence of matrices $B(\gg)^{n}$ will converge to zero 
as $n\to\infty$ on the subspace $E_-$ (that is, every vector $B(\gg)^{n}\ww$ with $\ww\in E_-$ converges to zero componentwise). 
Hence, every matrix norm will become contractive after a sufficiently large number $n$ of iterations.

However, what we state here above is that the contraction property holds for $n=N$, uniformly for large $N$, and for the specific norm induced by $\|\cdot\|_{\ell_1}$.

In conclusion, thanks to \eqref{ineq:ell-1-contraction}, we obtain a contractivity estimate for $n=N\to\infty$, that is for $T=1$. 
By iteration, as in the proof of \cite[Theorem 1, p.204]{A-A-DS2018}, one can deduce an exponentially decaying estimate, sketched as follows:

$\bullet$\quad for every integer $h\ge 1$ and every $t\in[h,h+1)$, one has
\begin{align*}
    \|J(\cdot,t\|_\infty \le \frac 1{2N} \tv \bar{J}_0 + \bigl\| B(\gg)^{hN}\bar \ww \bigr\|_{\ell_1} 
\end{align*}
where $\bar \ww$ is the projection of $\ssigma(0+)$ on $E_-$ and 
\begin{equation}
    \bar{J}_0: [0,1]\to \R\,,\qquad \bar{J}_0(x) = \begin{cases}
    {J}_0(x) & 0<x<1\\
    0  & x=0 \mbox{ or } x=1\,.
    \end{cases}
\end{equation}

$\bullet$\quad Therefore, by means of \eqref{ineq:ell-1-contraction}, one obtains
\begin{align*}
    \|J(\cdot,t\|_\infty &\le \frac 1{2N} \tv \bar{J}_0 + C_N(d)^{h}\bigl\|\bar \ww \bigr\|_{\ell_1} \\
    &\le \frac 1{2N} \tv \bar{J}_0 + C_N(d)^{-1} \ee^{- C t} \, \bigl\|\bar \ww \bigr\|_{\ell_1}
\end{align*}
for $N$ large enough so that $0<C_N(d)<1$\,, and $C=|\ln\{C_N(d)\}|$. 

We remark that the norm $\bigl\|\bar \ww \bigr\|_{\ell_1}$ depends on the total variation of the initial data (see \cite[p.205]{A-A-DS2018}); 
therefore the estimate above is not suitable to the extension to $L^\infty$ initial data.

\subsection{Contractivity of the invariant domain}\label{subsec:contractivity-inv-dom}
Next, under the assumptions \eqref{eq:gamma-costant}, we prove a contractivity property of the invariant domain $[m,M]^2$ for the approximate solutions.

\begin{proposition}\label{prop:invariance-wwd}
Given $\bar \ww\in\R^{2N}$ such that $\bar \ww\cdot \vv_{2N}=0$, and given $d\ge 0$, let
\begin{equation*}
\ww(d) = \bar \ww + \frac dN \left(1+\frac dN\right)^{-1}\Phi(\bar \ww)
\end{equation*}
where
\begin{equation}\label{def-of-Phi}
\Phi(\ww) = \left(\ww\cdot \vv_{2N}, - \ww\cdot \vv_2, \ww\cdot \vv_2,   \ldots,- \ww\cdot \vv_{2N-2},
\ww\cdot \vv_{2N-2},  - \ww\cdot \vv_{2N} \right)\,,\qquad \ww\in\R^{2N}
\end{equation}
for $\vv_{2\ell}$, $\ell=0,\ldots,N$ defined as in \eqref{def:v-ell}. Then one has
\begin{equation}\label{eq:inverse-of-Phi}
\bar \ww = \ww(d) - \frac d N \Phi(\ww(d))
\end{equation}
and
\begin{equation}\label{eq:inverse-of-Phi-B0N}
B(\zero)^{N}\bar \ww = B(\zero)^{N}\ww(d) - \frac d N \Phi(B(\zero)^{N}\ww(d))\,.
\end{equation}
Moreover, let $m\le 0\le M$ be such that
\begin{equation}\label{asump:bar-ww}
m\le \bar \ww\cdot \vv^\pm_{2\ell} \le M\qquad \ell=0,\ldots,N\,.
\end{equation}
Then one has, for every $d_1\ge 0$, $d>0$ and $j$, $k$:
\begin{align}
B(\ddone) \bar \ww \cdot  (\vv^\pm_{2j} - \vv^\pm_{2k})  &\le M-m\,,\label{eq:inv-domain-vector-form}\\
\ww(d)\cdot(\vv^\pm_{2j} - \vv^\pm_{2k}) & \le (1+d) (M-m)  
\,,\label{eq:estima-ww-d}\\
B(\ddone) \ww(d) \cdot  (\vv^\pm_{2j} - \vv^\pm_{2k})  &\le (1+d) (M-m)\,.\label{eq:estima-B0-ww-d}
\end{align}
\end{proposition}
\begin{proof} To prove \eqref{eq:inverse-of-Phi}, by the definition of $\ww(d)$, we need to prove that
\begin{equation}\label{eq:one}
\Phi(\ww(d)) = \left(1+\frac dN\right)^{-1}\Phi(\bar \ww) \,.    
\end{equation}
Thanks to the definition of $\vv_{2\ell}$\,,
$$
\vv_0 = \zero\,,\qquad \vv_{2\ell}=(\underbrace{1,\cdots,1}_{2\ell},0,\cdots,0)\qquad \ell=1,\ldots,N \,,
$$
we easily find that
\begin{equation*}
\Phi(\ww)\cdot \vv_{2\ell} = \sum_{j=1}^{2\ell} \Phi(\ww)_j = - \ww\cdot \vv_{2\ell}\,,\qquad \ell=1,\ldots,N\,.
\end{equation*}
Then we claim that the map $\Phi$ satisfies the following property: $$\Phi(\Phi(\ww)) = - \Phi(\ww)\,.$$ 
Indeed
\begin{align*}
\Phi(\Phi(\ww)) &=\left(
0, \underbrace{- \Phi(\ww)\cdot \vv_2}_{=\ww\cdot \vv_2}, \Phi(\ww)\cdot \vv_2,   \ldots,\underbrace{- \Phi(\ww)\cdot \vv_{2N-2}}_{=  \ww\cdot \vv_{2N-2}},
\Phi(\ww)\cdot \vv_{2N-2},  0 
\right)\\
&= - \Phi(\ww)\,.
\end{align*}
Since $\Phi$ is linear, one has
\begin{align*}
\Phi(\ww(d)) &= \Phi( \bar \ww) + \frac dN \left(1+\frac dN\right)^{-1}\underbrace{\Phi(\Phi(\bar \ww))}_{-\Phi(\bar \ww)}\\
& =  \Phi( \bar \ww)\left[1 -  \frac dN \left(1+\frac dN\right)^{-1}\right] 
= \left(1+\frac dN\right)^{-1}\Phi( \bar \ww)\,.
\end{align*}
This proves \eqref{eq:one} and hence \eqref{eq:inverse-of-Phi}. To prove \eqref{eq:inverse-of-Phi-B0N}, it is sufficient to prove that
\begin{equation}\label{eq:commut-Phi_B0N}
\Phi(B(\zero)^{N}\ww(d)) = B(\zero)^{N}\Phi(\ww(d))\,.
\end{equation}
Indeed, if \eqref{eq:commut-Phi_B0N} holds, from \eqref{eq:inverse-of-Phi} we find immediately that
\begin{equation*}
B(\zero)^{N}\bar \ww = B(\zero)^{N} \ww(d) - \frac d N  B(\zero)^{N}\Phi(\ww(d))  = B(\zero)^{N} \ww(d) - \frac d N  \Phi( B(\zero)^{N}\ww(d) )\,, 
\end{equation*}
hence  \eqref{eq:inverse-of-Phi-B0N} holds.

To prove \eqref{eq:commut-Phi_B0N}, let $\ww$ any vector in $\R^{2N}$ such that $ \ww\cdot \vv_{2N}=0$.  We recall \eqref{B-zero-N} to find that
\begin{align*}
B(\zero)^{N} \ww\cdot \vv_{2\ell} &=  \ww\cdot B(\zero)^{N} \vv_{2\ell} \\
&= \ww\cdot \left( \vv_{2N} - \vv_{2N-2\ell}\right) = \ww\cdot  \vv_{2N} -   \ww \cdot \vv_{2N-2\ell}\\
&= -  \ww \cdot  \vv_{2N-2\ell}
\end{align*}
and hence
\begin{equation*}
\Phi(B(\zero)^{N} \ww) = \left(0,  \ww\cdot \vv_{2N-2}, - \ww\cdot \vv_{2N-2},   \ldots, \ww\cdot \vv_{2},
- \ww\cdot \vv_{2},  0 \right) = B(\zero)^{N} \Phi( \ww)\,.
\end{equation*}
Since $\ww(d)\cdot \vv_{2N}=0$ for every $d\ge 0$, the previous identity applies and \eqref{eq:commut-Phi_B0N} holds. 

\smallskip
To prove \eqref{eq:inv-domain-vector-form}, recall \eqref{decompose-B}, then we have
\begin{align*}
B(\ddone) \bar \ww \cdot  (\vv^\pm_{2j} - \vv^\pm_{2k})=\frac{1}{1+d_1}\left( \underbrace{ B(\zero)\bar \ww \cdot  (\vv^\pm_{2j} - \vv^\pm_{2k})}_{(I)}+d_1 \underbrace{ B_1\bar \ww \cdot  (\vv^\pm_{2j} - \vv^\pm_{2k})}_{(II)}\right)
\end{align*}
Estimate of $(I)$, 
\begin{align*}
(I)= \bar \ww \cdot  B(\zero)^t (\vv^\pm_{2j} - \vv^\pm_{2k})\,,
\end{align*}
and one can check that the following holds true
\begin{align*}
B(\zero)^t (\vv^+_{2j} - \vv^+_{2k})&=\vv^+_{2j-2} - \vv^+_{2k-2}\\
B(\zero)^t (\vv^-_{2j} - \vv^-_{2k})&=\vv^+_{2j+2} - \vv^+_{2k+2}\,.
\end{align*}
Therefore, by \eqref{asump:bar-ww}, we get 
\begin{align*}
(I)= \begin{cases} \bar \ww\cdot(\vv^-_{2j-2} - \vv^-_{2k-2})\le M-m\\
 \bar \ww\cdot(\vv^+_{2j+2} - \vv^+_{2k+2})\le M-m
\end{cases}
\end{align*}
Estimate of $(II)$, one has the following
\begin{align*}
(II)&=\bar\ww\cdot B_1(\vv^\pm_{2j} - \vv^\pm_{2k})\\
&=- \bar\ww\cdot (\vv^\mp_{2j} - \vv^\mp_{2k})\\
&\leq M-m\,,
\end{align*}
the last inequality holds by \eqref{asump:bar-ww}. 
Hence, 
\begin{align*}
B(\dd) \bar \ww \cdot  (\vv^\pm_{2j} - \vv^\pm_{2k})\leq \frac{1}{1+d_1}\left((M-m)+d_1(M-m)\right)=M-m\,.
\end{align*}
The proof of \eqref{eq:inv-domain-vector-form} is complete.
To prove \eqref{eq:estima-ww-d}, one has that
\begin{equation*}
\ww(d)\cdot \vv^\pm_{2j} = \bar \ww\cdot\vv^\pm_{2j} + \frac dN \left(1+\frac dN\right)^{-1}\Phi(\bar \ww)\cdot \vv^\pm_{2j}\,,
\end{equation*}
where the map $\Phi$ satisfies 
\begin{align*}
\Phi(\bar \ww)\cdot \vv^-_{2j}&=\sum_{\ell=1}^{j} \bar \ww\cdot \vv_{2\ell} 
\\
\Phi(\bar \ww)\cdot \vv^+_{2j}&=\sum_{\ell=1}^{j} \bar \ww\cdot \vv_{2\ell-2} 
\,.
\end{align*}
By \eqref{asump:bar-ww} we find that
\begin{equation*}
\bar \ww\cdot \vv_{2\ell} = \bar \ww\cdot\left(\vv^+_{2\ell}-\vv^-_{2\ell}\right) \le M-m\,,
\end{equation*}
and hence we have
\begin{align*}
\ww(d)\cdot(\vv^-_{2j} - \vv^-_{2k})  
&=  \bar \ww\cdot (\vv^-_{2j} -  \vv^-_{2k}) + \frac dN \left(1+\frac dN\right)^{-1} \sum_{\ell=k+1}^{j} \bar \ww\cdot \vv_{2\ell}\\
& \le (M-m) + \frac dN \underbrace{\left(1+\frac dN\right)^{-1}}_{\le 1} \underbrace{(j-k)}_{\le N} \, (M-m)\\
& \le (M-m) \left( 1 + d \right)
\end{align*}
from which \eqref{eq:estima-ww-d} follows, in the case of the $v^-$ vectors. The estimate for $\ww(d)\cdot(\vv^+_{2j} - \vv^+_{2k})$ 
is completely similar and we omit it. 

The proof of \eqref{eq:estima-B0-ww-d} is a consequence of \eqref{eq:estima-ww-d} and is similar to the proof of \eqref{eq:inv-domain-vector-form}.
\end{proof}

\begin{theorem} \label{th:Linfty-d} Let $f^\pm$ be the approximate solution corresponding to the linear problem \eqref{eq:telegrapher}. Let $N\in 2\N$ and let $m\le 0 \le M$ be the constant values defined at \eqref{def:inv-dom}\,.

Then there exist constants $\CC_N(d)$ and $\widehat C>0$, such that 
\begin{align}\label{estim:fpm-time-h}
 \sup f^\pm(\cdot,t^{N})-\inf f^\pm(\cdot,t^{N}) &\leq \CC_N(d) (M-m)+  \frac{\widehat C}{N}\,.
\end{align}
\end{theorem}

\begin{proof} The proof employs the representation formula \eqref{eq:represent-fpm-linear} for $f^\pm$
and the expansion formula \eqref{eq:sigma-thN}.

\medskip\par
$\bullet$\quad We start from the representation formula \eqref{eq:represent-fpm-linear}. First we notice that
\begin{equation}\label{eq:fpm-xpm}
|f^\pm(x_j+,t) - f^\pm(x_j-,t)| \le \sup |J(\cdot,t)| \frac d N  \le (M-m)   \frac d N
\end{equation}
that vanishes as $N\to\infty$.

Since the $f^\pm$ are possibly discontinuous only at $x=x_j$ and along $(\pm1)$-- waves, then their image is given by the values at $x=0+$, $x=1-$ and $x=x_j\pm$ with $j=1,\ldots,N-1$. For this reason
in the following we will focus only the values of $f^\pm$ at $x=x_j+$, that is
\begin{equation}\label{eq:represent-fpm-1}
f^\pm(x_j+,t) = \ssigma(t) \cdot  \vv^\pm_{2j} +   \frac 12 \rho(0+,t) 
-   \frac d N \sum_{0\le \ell \le j } J(x_\ell,t) \,,\qquad j=0,\ldots,N-1
\end{equation}
and then we will use \eqref{eq:fpm-xpm} to conclude.

\medskip\par
$\bullet$\quad Let's rewrite the last sum in \eqref{eq:represent-fpm-1}. The identities \eqref{def:J(x,t)}--\eqref{rmk:vv} yield
$$
J(x_\ell,t) = 
\ssigma(t)\cdot \vv_{2\ell}\,.
$$
By the definition of $\Phi$ at \eqref{def-of-Phi},
\begin{equation*}
\Phi(\ssigma) = \left(0, - \ssigma\cdot \vv_2, \ssigma\cdot \vv_2,   \ldots,- \ssigma\cdot \vv_{2N-2},
\ssigma\cdot \vv_{2N-2},  0 \right)
\end{equation*}
and therefore
\begin{equation}\label{eq:represent-fpm-2}
\sum_{0\le \ell \le j } J(x_\ell,t) = \Phi(\ssigma(t))\cdot  \vv^-_{2j}= \Phi(\ssigma(t))\cdot  \vv^+_{2j+2}\,.
\end{equation}


\medskip\par
$\bullet$\quad Let $j$, $k \in \{0,\ldots, N-1\}$, $j>k$. We combine \eqref{eq:represent-fpm-1} and \eqref{eq:represent-fpm-2} to get
\begin{equation*}
\begin{aligned}
(a) \quad f^-(x_{j}+,t)- f^-(x_k+,t) & = \left(\ssigma(t) - \frac d N \Phi(\ssigma(t)) \right) \cdot  \left( \vv^-_{2j} -  \vv^-_{2k}\right) \\
(b) \quad f^+(x_{j}+,t)- f^+(x_k+,t) & = 
\ssigma(t)\cdot  \left( \vv^+_{2j} -  \vv^+_{2k}\right) - \frac d N \Phi(\ssigma(t)) \cdot  \left( \vv^+_{2j+2} -  \vv^+_{2k+2}\right)\,.
\end{aligned}
\end{equation*}
We claim that the following inequalities hold:
\begin{equation}\label{diffrence-f-pm}
f^\pm(x_{j}+,t)- f^\pm(x_k+,t) \le \left(\ssigma(t) - \frac d N \Phi(\ssigma(t)) \right) \cdot  \left( \vv^\pm_{2j} -  \vv^\pm_{2k}\right)
 + \frac{2d}N (M-m)\,.
\end{equation}
Indeed, from the identity $(a)$ above we immediately get \eqref{diffrence-f-pm} for the $"-"$. On the other hand, 
to prove \eqref{diffrence-f-pm} for the $"+"$ sign, it is enough to check that
\begin{equation*}
\left| \Phi(\ssigma(t)) \cdot  \left( \vv^+_{2j+2} - \vv^+_{2j}-  \vv^+_{2k+2} + \vv^+_{2k}\right)\right| \le 2 (M-m)\,,
\end{equation*}
which is true since 
$$
|\Phi(\ssigma(t)) \cdot  \left( \vv^+_{2j+2} - \vv^+_{2j}\right)| = |\ssigma(t)\cdot \vv_{2j}| = |J(x_j,t)|\le M-m\,.
$$
Therefore the claim is proved.

\smallskip\par\noindent
Next, we proceed with the analysis of the term
$$
\left(\ssigma(t) - \frac d N \Phi(\ssigma(t)) \right) \cdot  \left( \vv^\pm_{2j} -  \vv^\pm_{2k}\right) = (*)
$$
that appears in \eqref{diffrence-f-pm}.

By applying the identity \eqref{eq:sigma-thN},
the expression above can be written as a sum of three terms, 
corresponding to $B(\zero)^{N}$, $\widehat{P}$ and $R_{N}(d)$ respectively:
\begin{equation}\label{id_ABC}
(*) = \left(1+\frac d N \right)^{-N} \left[ \AA_1 + \AA_2 + \AA_3  \right]
\end{equation}
where
\begin{align*}
 \AA_1 & =  \left[B(\zero)^{N}\ssigma(0+)  - \frac d N \Phi\left( B(\zero)^{N}\ssigma(0+)\right)\right]  \cdot  (\vv^\pm_{2j} - \vv^\pm_{2k}) \\[1mm]
 \AA_2 & = {d}   \left[\widehat{P}\ssigma(0+)  - \frac d N \Phi\left( \widehat{P} \ssigma(0+)\right)\right] \cdot (\vv^\pm_{2j} - \vv^\pm_{2k})    \\[1mm]
 \AA_3 & = \left[ R_{N}(d)\ssigma(0+) - \frac d N \Phi\left( R_{N}(d)\ssigma(0+)\right)\right]    \cdot  (\vv^\pm_{2j} - \vv^\pm_{2k})\,.
\end{align*}

\medskip\par
$\bullet$\quad \textbf{Estimate for $\AA_2$.}\quad
We claim that 
\begin{equation*}
|\AA_2|\le \frac{d}N |\ssigma(0+)\cdot v_{-}|\,.
\end{equation*} 
To prove this claim, it is sufficient to prove that
\begin{equation*}
\begin{aligned}
(i)  &\qquad \widehat{P} \ssigma(0+) \cdot  (\vv^\pm_{2j} - \vv^\pm_{2k}) \in \{ \pm 1,0\}\,,
\\
(ii) &\qquad   \Phi\left(\widehat{P} \ssigma(0+)\right) \cdot  (\vv^\pm_{2j} - \vv^\pm_{2k})=0\,.
\end{aligned}
\end{equation*}

To prove $(i)$, we use \eqref{eq:hatPww-ssigma} to write that
\begin{align*}
\widehat{P} \ssigma(0+) \cdot \vv^\pm_{2\ell} =  \frac{1}{2N}\left(\ssigma(0+)\cdot v_{-}\right) 
\left(v_{-} \cdot \vv^\pm_{2\ell}\right) \,,
\end{align*}
where $v_-$ is the eigenvector in \eqref{v-pm}:
$$
v_{-}=(1,-1,-1,1,\ldots,1,-1,-1,1)\,. 
$$ 
From \eqref{def:vv-pm}, it is immediate to check that  
$$
v_{-}\cdot \vv^+_{2\ell}= (1,-1,-1,1,\ldots,1,-1,-1,1)\cdot  (\underbrace{1,0,\cdots,1,0}_{2\ell},0,\cdots,0)\in \{0, 1\}\,, 
$$
and similarly 
$$
v_{-}\cdot \vv^-_{2\ell}= - (1,-1,-1,1,\ldots,1,-1,-1,1)\cdot  (\underbrace{0,1,\cdots,0,1}_{2\ell},0,\cdots,0)\in \{0, 1\}\,. 
$$
More precisely, 
\begin{equation*}
v_{-}\cdot \vv^+_{2\ell}= v_{-}\cdot \vv^-_{2\ell} = \begin{cases}1&\mbox{ if } \ell \mbox{ odd}\\
0&\mbox{ if }  \ell  \mbox{ even}\,.
\end{cases}
\end{equation*}
Therefore, it is immediate to conclude that $(i)$ holds. 

To prove $(ii)$, we use the identity 
\begin{equation*} 
 \sum_{\ell=1}^j \ww\cdot \vv_{2\ell} = \Phi(\ww) \cdot  \vv^-_{2j} 
 \,,\qquad j\ge 1
\end{equation*}
that follows from the definition of $\Phi$ at \eqref{def-of-Phi},
to find that
\begin{align*}
 \Phi\left(\widehat{P} \ssigma(0+)\right)\cdot \vv^-_{2j} &= \sum_{\ell=1}^j \widehat{P} \ssigma(0+) \cdot \vv_{2\ell}\\
&= \frac{1}{2N}\left(\ssigma(0+)\cdot v_{-}\right)   \sum_{\ell=1}^j  \underbrace{v_{-} \cdot \vv_{2\ell}}_{=0}\\
 &=0\,.
\end{align*}
Here above we used the fact that $v_{-} \cdot \vv_{2\ell} = v_{-} \cdot\left( \vv^+_{2\ell} -  \vv^-_{2\ell}\right) =0$. 
The proof for 
$$\Phi\left(\widehat{P} \ssigma(0+)\right)\cdot \vv^+_{2j} = \sum_{\ell=1}^{j-1} \widehat{P} \ssigma(0+) \cdot \vv_{2\ell}$$ 
is totally analogous. The claim is proved.

\medskip\par
$\bullet$\quad  \textbf{Towards an estimate for $\AA_1$ and $\AA_3$.}
Consider the initial-boundary value problem with the same initial data and boundary condition as the one corresponding 
to $\ssigma(t)$, but for $k(x)\equiv 0$. Hence the problem is linear and undamped.

The corresponding evolution vector, that we denote with $\widehat \ssigma(t)$, is defined inductively by
\begin{equation}\label{def:widehat-sigma}
\begin{aligned}
\widehat \ssigma(t^n+) & = B(\zero)^{n} \widehat  \ssigma(0+)  
\,,\\
\widehat \ssigma(t^{n+\frac12}+) & = B_1 \widehat \ssigma(t^n+) \,,
\end{aligned}\qquad n\ge 1\,.
\end{equation}
About $\widehat \ssigma(0+)$ we claim that
\begin{equation}\label{init-data-k=0}
\widehat \ssigma(0+) = \ssigma(0+)  - \frac d N \Phi(\ssigma(0+))  
\end{equation}
where 
\begin{align*}
\Phi(\ssigma(0+))  &=  \left(0, - \ssigma(0+)\cdot \vv_2, \ssigma(0+)\cdot \vv_2,   \ldots,- \ssigma(0+)\cdot \vv_{2N-2},
\ssigma(0+)\cdot \vv_{2N-2},  0\right) \\
&=  \left(0, - J(x_{1},0+),  +J(x_{1},0+), \ldots,- J(x_{N-1},0+), + J(x_{N-1},0+), 0\right)^t\,.
\end{align*}
\smallskip\par\noindent
To prove the claim, 

- we observe that $\widehat\sigma_1= \sigma_1$ and $\widehat\sigma_{2N}= \sigma_{2N}$, it is obvious since $\Phi_1(\sigma (0+))=0$ and $\Phi_{2N}(\sigma (0+))=0$.

- at every $x_j$, $j=1,\ldots,N-1$ we compare $(\widehat\sigma_{2j}, \widehat\sigma_{2j+1})$ with $(\sigma_{2j}, \sigma_{2j+1})$.
In the notation of Proposition~\ref{prop:multiple}, let $J_*$ the middle value for $J$ in the solution to the Riemann problem with $d=\bar k>0$ and
$J_m= f^+_\ell - f^-_r$ the middle value for $J$ when $\bar k=0$. Using \eqref{J*_rho*}, we have the following identity:
\begin{equation*}
J_* + \frac dN  J_* = J_m\,,
\end{equation*}
from which we deduce
\begin{equation*}
\widehat\sigma_{2j} = J_m - J_\ell = (\underbrace{J_* - J_\ell}_{=\sigma_{2j}}) + \frac dN  J_* = \sigma_{2j} +  \frac dN J(x_{j},0+)\,.
\end{equation*}
Similarly one has
\begin{equation*}
\widehat\sigma_{2j+1} = J_r - J_m = (\underbrace{J_r - J_*}_{=\sigma_{2j+1}}) - \frac dN  J_* = \sigma_{2j+1} -  \frac dN J(x_{j},0+)\,.
\end{equation*}
Therefore \eqref{init-data-k=0} holds. The claim is proved.

\smallskip\par\indent
It is easy to check that \eqref{init-data-k=0} can be inverted as follows:
\begin{equation*}
\ssigma(0+) = \widehat \ssigma(0+)  + \frac d N  \left(1+\frac dN \right)^{-1}\Phi(\widehat\ssigma(0+))  \,,
\end{equation*}
see Proposition~\ref{prop:invariance-wwd}.

\medskip\par
$\bullet$\quad \textbf{Estimate for $\AA_1$.}\quad We apply \eqref{eq:inverse-of-Phi-B0N} to find that
\begin{align*}
 \AA & =  B(\zero)^{N} \widehat \ssigma(0+) \cdot  (\vv^\pm_{2j} - \vv^\pm_{2k})  \le M-m\,.
\end{align*}

\medskip\par
$\bullet$\quad \textbf{Estimate for $\AA_3$.}\quad 
By using \eqref{eq:R_n} we get 
\begin{align*}
& R_{N}(d)\ssigma(0+) - \frac d N \Phi\left( R_{N}(d)\ssigma(0+)\right) \\
& = \sum_{j=0}^{N-1}   \zeta_{j,N} \left\{ B_1 B(\zero)^{N-2j-1} \ssigma(0+)- \frac d N   \Phi\left(B_1 B(\zero)^{N-2j-1} \ssigma(0+)\right) \right\}\\
&\qquad 
 + \sum_{j=1}^{N-1}  \eta_{j,N} \left\{ B(\zero)^{2j-N} \ssigma(0+)- \frac d N   \Phi\left(B(\zero)^{2j-N} \ssigma(0+)\right) \right\}\,.
\end{align*}

By \eqref{eq:estima-B0-ww-d} for $d_1=0$, we have 
\begin{align*}
B(\zero)^{n} \ssigma(0+)- \frac d N \Phi\left(B(\zero)^{n} \ssigma(0+)\right)&\leq (1+d)(M-m) + d (1+d)(M-m)\\
&=(1+d)^2(M-m)
\end{align*}
%
The same hold for the term containing $B_1$. Therefore, by \eqref{stima-su-zeta-eta_jN},
\begin{align*}
\AA_3\leq \left(1+d\right)^2 \left(\ee^{d}  - d -1+ {\frac {K} N}\right)\left(M-m\right) 
\end{align*}
Finally, by recalling \eqref{id_ABC} and 
collecting the bounds on the terms $\AA_1$, $\AA_2$ and $\AA_3$, and using \eqref{eq:sigma-dot-v_-}, we get 
\begin{align*}
(*)&\leq \CC_N(d) \left(M-m\right) +  \frac{d}{N}\left(1+\frac d N \right)^{-N} \tv \bar{J}_0
\end{align*}
where
\begin{equation}\label{def:hat-CN(h,d)}
\CC_N(d)\,\dot{=}\, \left(1+\frac d N \right)^{-N}\left(1+(1+d)^2\left(\ee^{d}-d-1+{\frac {K} N}\right)\right)\,.
\end{equation}
In conclusion, combining the estimate above with \eqref{eq:fpm-xpm} and \eqref{diffrence-f-pm}, we conclude that
\begin{align*} 
0&\leq \sup f^\pm(\cdot,t^{N})-\inf f^\pm(\cdot,t^{N}) \leq \CC_{N}(d) (M-m)+  \frac{\widehat{C}}{N}
\end{align*}
for $\widehat{C}$ that can be chosen to be independent on $N$ as follows:
$$\widehat{C}= d \left[
\tv \bar{J}_0 + 3(M-m)\right]\,.$$
The proof of Theorem~\ref{th:Linfty-d} is now complete.
\end{proof}

Now we are ready to complete the proof of Theorem~\ref{main-theorem-2}.

\begin{proof} The proof of Theorem~\ref{main-theorem-2} is a consequence of \eqref{estim:fpm-time-h} in Theorem~\ref{th:Linfty-d}.

Indeed, given $(f^\pm)^\DX$, the convergence of a subsequence towards $f^\pm$ holds in $L^1(I)$ for all $t>0$ and hence, possibly up to a subsequence, almost everywhere. Hence we can pass to the limit in \eqref{estim:fpm-time-h} and get that
\begin{equation*}
 \esssup f^\pm(\cdot,1)-\essinf f^\pm(\cdot,1) \leq \CC(d) (M-m)    
\end{equation*}
where 
\begin{equation}\label{def:CC}
 \CC_{N}(d) \quad \rightarrow \quad \ee^{-d}\left(1+(1+d)^2\left(\ee^{d}-d-1\right)\right)=:\CC(d)\,,\qquad N\to\infty \,.   
\end{equation}
Since $\CC(0)=1$, $\CC'(0)=-1$ and $\CC(d)\to+\infty$ as $d\to+\infty$, then there exists a value $d^*>0$ such that $\CC(d^*)=1$ and
\begin{equation}\label{CC-less-than-1}
    0<\CC(d)<1\,,\qquad 0<d<d^*\,.
\end{equation}
This completes the proof of \eqref{eq:contraction-M1m1} for initial data $(\rho_0,J_0)\in BV(I)$.

On the other hand, if $(\rho_0,J_0)\in L^\infty(I)$, then there exists a sequence $(\rho_{0,n},J_{0,n})\in BV(I)$
that converges to $(\rho_0,J_0)$ in $L^1(I)$, and hence the limit solution satisfies the same $L^\infty$ bounds. Therefore \eqref{eq:contraction-M1m1} holds. 
The proof of Theorem~\ref{main-theorem-2} is complete.
\end{proof}

\section{Proof of Theorem~\ref{main-theorem-3-applications}}\label{sec:6} 
\setcounter{equation}{0}
In this section we prove Theorem~\ref{main-theorem-3-applications}, by employing the contracting estimate established in 
Theorem~\ref{main-theorem-2}. For the system
\begin{equation*}
\begin{cases}
\partial_t\rho +  \partial_x J  = 0, &\\
\partial_t J  +  \partial_x \rho = - 2 d \alpha(t) J, &
\end{cases}
\end{equation*}
we consider the following two situations: 

\textbf{(a)} $\alpha(t)\equiv 1$ 

\textbf{(b)} $\alpha(t)$ as in \eqref{hyp-on-alpha_ON-OFF} with $T_1\ge 1$\,. 

\par\noindent
Let's examine each one in detail.



\medskip\par\noindent

\textbf{(a)} 
In this case, we start by observing that the invariant domain property in Theorem~\ref{theorem:well-posedness} holds also for every $\bar t>0$: if 
\begin{equation*}
M(\bar t)= \esssup_I{f^\pm(\cdot,\bar t)}\,,\qquad m(\bar t) = \essinf_I{f^\pm(\cdot,\bar t)}\,,\qquad \bar t>0
\end{equation*}
then 
\begin{equation*}
    m(\bar t)\le f^\pm (x,t)\le M (\bar t) \qquad \mbox{for }a.e.\, x,\quad t> \bar t\,,
\end{equation*}
and the functions $- m(\bar t)$, $M(\bar t)$ are monotone non-increasing.

Let's define $M_0 = M$, $m_0=m$ and, for $h\in \N$,
\begin{equation*}
    M_h= \esssup_I{f^\pm(\cdot,h)}\,,\qquad m_h = \essinf_I{f^\pm(\cdot,h)}\qquad h\ge 1\,.
\end{equation*}
By the monotonicity property above, the two sequences satisfy 
\begin{equation}\label{m-M-n-monotone}
    m_0\le m_1 \le \ldots \le 0\le \ldots \le M_1\le M\,.
\end{equation}
We claim that the two sequences converge both to 0. Indeed,
by applying \eqref{eq:contraction-M1m1} iteratively, we obtain 
\begin{equation*}
    M_h- m_h \le \CC(d) \left(M_{h-1}- m_{h-1}\right)\,,\qquad h\ge 1
\end{equation*}
and therefore
\begin{equation}\label{decay-Mmn}
    M_h - m_h \le \CC(d)^h \left(M- m\right)\,,\qquad h\ge 1\,.
\end{equation}
Hence, by means of \eqref{m-M-n-monotone} and recalling that $\CC(d)<1$, we conclude that $M_h$ and $m_h\to 0$ as $h\to\infty$\,.

Therefore we obtain the bound
\begin{equation*}
    m_h\le f^\pm (x,t)\le M_h \qquad \mbox{for }a.e.\, x,\quad t\in [h,h+1)\,,
\end{equation*}
Recalling the relation \eqref{diag-var} between $\rho$, $J$ and $f^\pm$, we find that
\begin{align*}
    |J(x,t)| & = |f^+ (x,t)- f^-(x,t)| \le M_h-m_h\,,\\
    |\rho (x,t)| & = |f^+ (x,t)+ f^-(x,t)| \leq 2 \max\{M_h,|m_h|\}\leq 2( M_h-m_h)
\end{align*}
for $t\in [h,h+1)$\,.

Now we observe that one has, for $h\le t < h+1$:
\begin{equation*}
    \CC(d)^h < \CC(d)^{t-1} =  \frac 1 {\CC(d)} \ee^{-C_3 t}
\end{equation*}
where
\begin{equation*}
    C_3 = |\ln\left(\CC(d)\right)|\,.
\end{equation*}
Therefore, if we define
\begin{equation}\label{def-C1-C2}
    C_1 = \frac {M-m} {\CC(d)}\,,\qquad C_2 = 2 C_1
\end{equation}
and use \eqref{decay-Mmn}, we obtain  
\begin{align*}
\|J(\cdot,t)\|_{L^\infty}&\leq C_1 \ee^{ - C_3 t}\,,  \\
\|\rho(\cdot,t)\|_{L^\infty}&\leq C_2 \ee^{ - C_3 t}
\end{align*}
which is \eqref{decay-J-rho}\,. Hence the proof of part \textbf{(a)} is complete.

\medskip\par
\textbf{(b)} In this case, recalling \eqref{hyp-on-alpha_ON-OFF}, for $0<T_1<T_2$ one has 
\begin{equation*}
\alpha(t) =\begin{cases}
1 & t\in[0,T_1), \\
0 \, & t\in [T_1,T_2) 
\end{cases}
\end{equation*}
and $\alpha(t)$ is $T_2$-periodic\,. Therefore the damping term is "active" in every time interval of the form $[hT_2, hT_2+T_1)$ with $h\in\N$. 

Here we are assuming that $T_1\ge 1$. For $h\in \N$, define
\begin{equation*}
    M_h= \esssup_I{f^\pm(\cdot,hT_2)}\,,\qquad m_h = \essinf_I{f^\pm(\cdot,h T_2)}\qquad h\ge 1\,.
\end{equation*}
As in \emph{(a)}, by applying \eqref{eq:contraction-M1m1} iteratively, we obtain for $h\ge 1$
\begin{equation*}
    M_h- m_h \le \CC(d)^{[T_1]} \left(M_{h-1}- m_{h-1}\right)  
    \,.
\end{equation*}
Therefore
\begin{equation}\label{decay-Mmn-T1}
    M_h - m_h \le \CC(d)^{h[T_1]} \left(M- m\right)\,,\qquad h\ge 1\,.
\end{equation}
If $hT_2\le t <(h+1)T_2$, then
\begin{equation*}
    \CC(d)^{h[T_1]} = \CC(d)^{ (h+1)[T_1] - [T_1]} < \CC(d)^{- [T_1]} 
    \CC(d)^{\frac{[T_1]}{T_2} t} =  \frac 1 {\CC(d)^{[T_1]} } \ee^{-C_3 t}
\end{equation*}
with
\begin{equation*}
    C_3 = \frac{[T_1]}{T_2} |\ln\left(\CC(d)\right)|\,.
\end{equation*}
Proceeding as in  \emph{(a)} we obtain 
\begin{align*}
\|J(\cdot,t)\|_{L^\infty}&\leq C_1 \ee^{ - C_3 t}\,,  \\
\|\rho(\cdot,t)\|_{L^\infty}&\leq C_2 \ee^{ - C_3 t}
\end{align*}
with 
\begin{equation*}
    C_1 = \frac {M-m} {\CC(d)^{[T_1]}}\,,\qquad C_2 = 2 C_1\,.
\end{equation*}
The proof of part \textbf{(b)} is complete, and hence the proof of Theorem~\ref{main-theorem-3-applications}.

\appendix
\section{Proof of Theorem~\ref{theo:exp-formula}}\label{appendix:proof-of-exponential-formula}
\setcounter{equation}{0}
In this Appendix we prove Theorem~\ref{theo:exp-formula}. The expansion of the following power gives
\begin{equation}\label{eq:expansion-2}
\left[B(\zero)+ \gamma B_1 \right]^{n} =  \sum_{k=0}^{n} \gamma^k  S_k(B(\zero),B_1), 
\end{equation}
where each term $S_k(B(\zero),B_1)$ is the sum of all products of $n$ matrices which are either $B_1$ or $B(\zero)$,
and in which $B_1$ appears exactly $k$ times, that is
\begin{equation}\label{def:S_k}  
\left\{
\begin{aligned}
S_k(B(\zero),B_1)= & \sum_{(\ell_1,\ldots,\ell_{k+1})} 
B(\zero)^{\ell_1} \cdot B_1 \cdot B(\zero)^{\ell_2} \cdot B_1 \cdots B(\zero)^{\ell_k} \cdot B_1 \cdot B(\zero)^{\ell_{k+1}}\\
&0\le \ell_j\le n-k\,,\qquad   \sum_{j=1}^{k+1} \ell_j = n -k\,.
\end{aligned}\right.
\end{equation}
The terms $S_k$ can be handled, as in \cite{A-A-DS2018}, by means of the following identity:
\begin{equation}\label{commut-rule}
B(\zero)^{\pm\ell} B_1 = B_1  B(\zero)^{\mp\ell} \qquad \forall \, \ell\in\N\,.
\end{equation}
By means of \eqref{commut-rule} and using that $B_1^2=\Id$, the generic term $S_k$ 
in \eqref{def:S_k} can be conveniently rewritten: for $k=1,3,\dots, n-1$ odd we have 
\begin{equation}\label{eq:Skodd}
S_k(B(\zero),B_1) =  \sum_{j=\frac{k-1}2 }^{n-\frac{k+1}2}   
\begin{pmatrix} j \\ \frac{k-1}2     \end{pmatrix} \begin{pmatrix} n-j-1\\ \frac{k-1}2     
\end{pmatrix} B(\zero)^{2j-n}  B_2(\zero)
\end{equation}
and for $k=2,4,\dots, n$ even we have
\begin{equation}\label{eq:Skeven}
S_k(B(\zero),B_1)
=  \sum_{j=\frac k2 }^{n-\frac k2}  
\begin{pmatrix} j \\ \frac k2      \end{pmatrix} \begin{pmatrix} n-j-1\\ \frac k2 -1      \end{pmatrix}
B(\zero)^{2j-n}  \,. 
\end{equation}
In \eqref{eq:Skodd}, it is convenient to rewrite the term $B(\zero)^{2j-n}  B_2(\zero)$ as follows. Recalling that
$B(\zero)$ is given by $B(\zero)= B_2(\zero) B_1$, we obtain 
\begin{equation*}
B_2(\zero) = B_2(\zero)  B_1^2 = B(\zero) B_1 
\end{equation*}
and hence, by means of \eqref{commut-rule},
\begin{equation*}
B(\zero)^{2j-n}  B_2(\zero) = B(\zero)^{2j-n+1} B_1 = B_1 B(\zero)^{n-2j-1}\,.
\end{equation*}
Therefore, we can write \eqref{eq:expansion-2} for any $n$ as the following
\begin{align}
\left[B(\zero)+ \gamma B_1 \right]^{n} &= B(\zero)^{n} + ~\gamma  \sum_{j=0 }^{n-1}  B_1 B(\zero)^{n-2j-1} 
\label{eq:id-for-n-line1}\\
&\qquad   + \sum_{j=0}^{n-1}   \zeta_{j,n} B_1 B(\zero)^{n-2j-1} + \sum_{j=1}^{n-1}  \eta_{j,n} B(\zero)^{2j-n}\,, \nonumber
\end{align}
where $\gamma=\frac d N$ and 
\begin{align}
&\zeta_{j,n}  = \sum_{\ell=1}^{\min\{j,n-j-1\}}\gamma^{2\ell +1} \begin{pmatrix} j \\ \ell
\end{pmatrix} \begin{pmatrix}
n-j-1 \\ \ell
\end{pmatrix}\label{def:binomial-coeff-zeta}\,,
\\
&\eta_{j,n} ~=  \sum_{i=1}^{\min\{j,n-j\}}\gamma^{2i} 
\begin{pmatrix}
j \\ i
\end{pmatrix} \begin{pmatrix}
n-j -1 \\ i -1
\end{pmatrix}\label{def:binomial-coeff-eta}.
\end{align}
In the expansion above, the term with the $\zeta_{j,n}$ accounts for the odd powers, $\ge 3$, of $\gamma$ while the term 
with the $\eta_{j,n}$ accounts for the even powers $\ge 2$ of $\gamma$.

\smallskip
From now on, we assume that $n=N$. We recall the identity \cite[(100)]{A-A-DS2018},
\begin{equation}\label{eq:sum-full-cycle}
\frac{1}{N}\sum_{j=0}^{N-1} B(\zero)^{2j} 
= \frac {1}{2N} \left(e^t e  + v_-^t v_-  \right) = \widehat P\,, 
\end{equation}
and some immediate identities,
\begin{equation*}
\widehat P B_2(\zero) 
= \widehat P\,,\qquad  B(\zero)^{2}  \widehat P =  \widehat P B(\zero)^{2} =  \widehat P\,.
\end{equation*}
Therefore
\begin{equation*}
\sum_{j=0 }^{N-1} B_1 B(\zero)^{N-2j-1} =  B_1 \sum_{j=0 }^{N-1} B(\zero)^{N-2j-1} = N \widehat P\,,
\end{equation*}
and the identity \eqref{eq:id-for-n-line1} rewrites as
\begin{align*}
\left[B(\zero)+ \gamma B_1 \right]^{N} &= B(\zero)^{N} + d \widehat P +  R_{N}(d) \\
R_{N}(d) & 
=  \sum_{j=0}^{N-1}   \zeta_{j,N}  B_1 B(\zero)^{N-2j-1} + \sum_{j=1}^{N-1}  \eta_{j,N} B(\zero)^{2j-N}\,.
\end{align*}
To complete the proof, we need to estimate the sums of  $\zeta_{j,N}$, $\eta_{j,N}$. We claim that 
\begin{align}\label{stima-su-zeta_jN}
0\le \sum_{j=0}^{N} \zeta_{j,N} &\le \sinh(d)  - d+ {\frac1N} f_0(d)  ~~~   
\\
0\le \sum_{j=1}^{N} \eta_{j,N} &\le  \cosh(d)-1  ~+~  {\frac1N} f_1(d) ~~~  
\label{stima-su-eta_jN}
\end{align}
where
\begin{align*}
f_0(d) &~\dot =~ \sum_{\ell=1}^{\infty} \left(\frac{1}{2}\right)^{2\ell}\frac{d^{2\ell+1}}{(\ell!)^2}  = d \left[ I_0(d) -1 \right]\\
f_1(d) &~\dot =~\sum_{i=1}^{\infty} \left(\frac{1}{2}\right)^{2i-1}\frac{(d)^{2i}}{i! (i-1)! } = d I_1(d)\,,   
\end{align*}
and
$$
I_\alpha(2x) = \sum_{m=0}^\infty \frac{x^{2m+\alpha}}{m! (m+\alpha)!} \,,\qquad \alpha=0,1
$$
is a modified Bessel function of the first type. It is clear that, once that the claim above is proved, then it follows that 
\eqref{stima-su-zeta-eta_jN} holds with 
\begin{equation}\label{K-sum-modif-bessel-0}
K(d)=  f_0(d) +  f_1(d)\,.
\end{equation}

\smallskip
We start with $\zeta_{j,N}$ defined in \eqref{def:binomial-coeff-zeta}. Using the inequality 
$$
\begin{pmatrix} n \\ k
\end{pmatrix} \le \frac{n^k}{k!}\,,\qquad 0\le k\le n
$$
and the definition $\gamma=d/N$, we find that
\begin{align}\label{ineq:zeta_jN}
\zeta_{j,N}&\le 
\frac{1}{N}\sum_{\substack{\ell=1}}^{\infty}\frac{(d)^{2\ell+1}}{(\ell!)^2} \frac{ j^\ell}{N^\ell}  \frac{(N-j-1)^\ell}{N^\ell} \,.
\end{align}  
Then we introduce the change of variable
\begin{equation}\label{def:x-j}
x_j =  \frac {j} {N}\,,\qquad j=0,\ldots, N-1\,.
\end{equation}
Thanks to the inequality \eqref{ineq:zeta_jN} we get
\begin{align*}
0\le \zeta_{j,N}&\le \frac {1} {N}\sum_{{\ell=1}}^{\infty}\frac{(d)^{2\ell+1}}{(\ell!)^2 }x_j^\ell  \left(1-x_j- \frac {1} {N}\right)^\ell \\
&\le \frac {1}{N} \sum_{{\ell=1}}^{\infty}\frac{(d)^{2\ell+1}}{(\ell!)^2 } x_j^{\ell}{(1-x_j  )}^\ell \,.
\end{align*}  
As a consequence, we deduce an estimate for the sum of the $\zeta_{j,N}$:
\begin{align*} 
0\le \sum_{j=0}^{N-1} \zeta_{j,N} &\le {\frac {1} {N}} \sum_{j=0}^{N-1} \sum_{{\ell=1}}^{\infty}\frac{d^{2\ell+1}}{(\ell!)^2} x_j^{\ell}{(1-x_j  )}^\ell \\
&=  \sum_{\ell=1}^{\infty} \frac{d^{2\ell+1} }{(\ell!)^2 } \left\{\frac 	{1}{N} \sum_{j=0}^{N-1} x_j^{\ell}{(1-x_j  )}^\ell  \right\}
\end{align*}
Using the definition \eqref{def:x-j}, we observe that
\begin{equation*}
\frac{1}{N}\sum_{j=0}^{N-1} x_j^{\ell}{(1-x_j  )}^\ell  ~~\to~~ \int_{0}^1 x_j^{\ell}{(1-x_j  )}^\ell \,dx \quad \mbox{as}\  {N\to\infty}, \qquad \ell\ge 1\,;
\end{equation*}
more precisely the following estimate holds,
\begin{align}\nonumber
\frac{1}{N}\sum_{j=0}^{N-1}  x_j^{\ell}{(1-x_j  )}^\ell & = \frac{1}{N}\left(\sum_{j=0}^{(N/2)-1} + \sum_{j=(N/2)+1}^{N-1}\right) x_j^{\ell}{(1-x_j )}^\ell
~+~\frac{1}{N}\left(\frac{1}{2}\right)^{2\ell}\\
&\le \int_{0}^1  x_j^{\ell}{(1-x_j )}^\ell\,dx ~+~ \frac{1}{N}\left(\frac{1}{2}\right)^{2\ell} \,. \label{int-bessel}
\end{align}
It is easy to check the following identities
\begin{equation}\label{integrals}
\int_{0}^1  x_j^{\ell}{(1-x_j )}^\ell\,dx = \frac{(\ell!)^2} {(1+2\ell)!}\,, \qquad \ell\ge 1\,.  
\end{equation}
By plugging the previous estimates into the sum of the $\zeta_{j,n}$ we get
\begin{align*} 
0\le \sum_{j=0}^{N-1} \zeta_{j,n} 
& \leq \sum_{\ell=1}^{\infty}\frac{d^{2\ell+1}}{(\ell!)^2}  \frac{(\ell!)^2} {(1+2\ell)!} ~+~  
\frac {1} {N} \underbrace{\sum_{\ell=1}^{\infty} \left(\frac{1}{2}\right)^{2\ell}\frac{d^{2\ell+1}}{(\ell!)^2}}_{= f_0(d)}\\
&=\sum_{\ell=1}^{\infty} \frac{d^{2\ell+1}} {(1+2\ell)!} ~+~  \frac{1}{N} f_0(d) \\
&=\sinh(d)  -d ~+~  \frac{1}{N}  f_0(d)\,.
\end{align*}
Therefore \eqref{stima-su-zeta_jN} follows.

\smallskip
Similarly to the estimate \eqref{ineq:zeta_jN} for $\zeta_{j,N}$ and using the change of variables \eqref{def:x-j}, for $\eta_{j,N}$ defined in \eqref{def:binomial-coeff-eta} we find that
\begin{align*}
\eta_{j,N} &\le   \frac {1}{N}  \sum_{i=1}^{\infty}\frac{d^{2i}}{i! (i-1)! }  x_j^i \left(1-x_j - \frac {1}{N} \right)^{i-1} \\
&\le  \frac {1}{N}  \sum_{i=1}^{\infty}\frac{d^{2i}}{i! (i-1)! }  x_j^i \left(1-x_j \right)^{i-1}\,.
\end{align*}
The sum of the $\eta_{j,N} $  can be estimated as follows,
\begin{align*}
\sum_{j=1}^{N-1} \eta_{j,N} &\le   \sum_{i=1}^{\infty}\frac{d^{2i}}{i! (i-1)! }  \left\{ \frac {1}{N} \sum_{j=1}^{N-1}  x_j^i \left(1-x_i \right)^{i-1}  \right\} \,.
\end{align*}
while by \eqref{int-bessel} with $\ell=i-1$ and by \eqref{integrals} we find that
\begin{align*}
\frac {1}{N}  \sum_{j=1}^{N-1} x_j^i \left(1-x_j \right)^{i-1}&\le  \int_{0}^1  x_j^i \left(1-x_j \right)^{i-1}\,dx ~+~ \frac {1}{N} \left(\frac{1}{2}\right)^{2i-1} \\
& =  \frac{(i-1)!(i)!} {(2i)!}  ~+~ \frac {1}{N} \left(\frac{1}{2}\right)^{2i-1}\,.
\end{align*}
Therefore 
\begin{align*}
\sum_{j=1}^{N-1} \eta_{j,N} &\le  \sum_{i=1}^{\infty}\frac{d^{2i}}{i! (i-1)! } \frac{(i-1)!(i)!} {(2i)!}  ~+~\frac {1}{N}
\underbrace{\sum_{i=1}^{\infty} \left(\frac{h}{2}\right)^{2i-1}\frac{d^{2i}}{i! (i-1)! }}_{= f_1(d)}\\
&= \sum_{i=1}^{\infty}\frac{d^{2i}}{(2i)!}   ~+~ \frac 1 N {f_1(d)} \\
&= \cosh(d)-1  ~+~\frac 1 N {f_1(d)}\,,
\end{align*}
that leads to \eqref{stima-su-eta_jN}. This completes the proof of Theorem~\ref{theo:exp-formula}. 


\medskip
\medskip

\end{document}